 \documentclass[11pt,letterpaper]{amsart}

\usepackage{amssymb,amsmath,pgf}
\usepackage{tikz}
\usetikzlibrary{arrows,patterns}

\newtheorem{theorem}{Theorem}[section]
\newtheorem{corollary}[theorem]{Corollary}
\newtheorem{definition}[theorem]{Definition}
\newtheorem{lemma}[theorem]{Lemma}
\newtheorem{proposition}[theorem]{Proposition}
\newtheorem{remark}[theorem]{Remark}

\newtheorem{example}[theorem]{Example}

\newenvironment{proofof}[1]{\noindent{\it Proof of
#1.}}{\hfill$\square$\\\mbox{}}

\def\qui{Q}
\def\ugqui{Q}
\def\source{\mathrm{s}} 
\def\target{\mathrm{t}}

\def\F{\mathbb{F}}
\def\rep{\mathcal{R}}
\def\semi{\mathcal{S}}
\def\moduli{\mathcal{M}}

\def\Z{\mathrm{Z}}
\def\GL{\mathrm{GL}} 
\def\SL{\mathrm{SL}}

\begin{document}
 
\title{Quiver moduli spaces of a given dimension}
\author{{M\'aty\'as Domokos}}
\thanks{Partially supported by NKFIH K 138828 and K 132002. }
\address{HUN-REN Alfr\'ed R\'enyi Institute of Mathematics,
Re\'altanoda utca 13-15, 1053 Budapest, Hungary,
ORCID iD: https://orcid.org/0000-0002-0189-8831}
\email{domokos.matyas@renyi.hu}

\subjclass[2010]{Primary 14L24; Secondary 13A50, 16G20} 
\keywords{moduli space of quiver representations, semi-invariants of quivers}
\date{}
\maketitle 


\begin{abstract} 
It is shown that certain transformations on quiver-dimension vector pairs induce 
isomorphisms on the corresponding moduli spaces of quiver representations and 
map a stable dimension vector to a stable dimension vector. 
This result combined with a combinatorial analysis of dimension vectors in the fundamental 
set of a wild quiver is applied to prove that in each dimension there are only finitely many 
projective algebraic varieties occurring as a 
moduli space of representations of a quiver with a dimension vector that 
satisfies some simple constraints. 
\end{abstract}

\section{Introduction} \label{sec:intro} 

The moduli spaces introduced by King \cite{king} provide a link between 
algebraic geometry and representation theory,  by which quiver representations are 
involved in several areas of mathematics. Quiver Grassmannians give another construction associating projective algebraic varieties to quiver representations. 
Every projective algebraic variety can be realized as a 
fine moduli space of representations of a quiver with relations (see \cite{hille:1}, 
\cite{huisgen-zimmermann:1}) and also as a quiver Grassmannian 
(see \cite{huisgen-zimmermann:2}, \cite{reineke}). 
As it is remarked in \cite{ringel:1}, these results can be traced back to \cite{beilinson}, 
and have several variants and refinements, see 
\cite{bongartz_huisgen-zommermann}, \cite{craw-smith}, \cite{derksen_huisgen-zimmermann_weyman},   \cite{hille:2}, 
\cite{ringel:1}, \cite{ringel:2}, \cite{ringel:3}  for a sample.

 On the other hand, the class of moduli spaces of representations of quivers (with no relations) 
 is more exclusive in this respect. For example, as it is observed in 
 \cite[Page 4]{fei}, the only $1$-dimensional projective variety that occurs as a moduli space of representations of some quiver is the projective line $\mathbb{P}^1$. 
Developing further  some earlier works in \cite{altmann-straten} and \cite{altmann-etal}, it was pointed out in \cite{domokos-joo} that in each dimension there are only finitely many toric varieties that can be realized as the moduli space of thin representations of quivers, and a  procedure for their classification is outlined. 
Here "thin" means that the dimension vector is constant $1$, in which case the resulting moduli spaces have a canonical structure of a toric variety. 
Note also that the number of isomorphism classes of algebraic varieties that occur as a moduli space of representations of a quiver is countable. Indeed, the number of 
possible quiver-dimension vector pairs is countable, and for a fixed quiver and a fixed dimension vector there are finitely many GIT-equivalence classes of weights, hence there are finitely many possible moduli spaces.

The present paper is motivated by the following questions: 
how many $d$-dimensional projective algebraic varieties occur as 
a moduli space of representations of a quiver for a given $d$, and how large quivers and dimension vectors are needed to realize them? 

Recall that the moduli spaces $\moduli(\qui,\alpha,\theta)$ defined in  \cite{king} are quasi-projective varieties depending on a quiver $\qui$, a dimension vector $\alpha$, and a weight $\theta$. Under certain conditions on $(\qui,\alpha)$, we shall define  
$(\qui',\alpha')$ and a weight $\theta'$ such that 
$\moduli(\qui,\alpha,\theta)\cong \moduli(\qui',\alpha',\theta')$; moreover, if $\alpha$ is $\theta$-stable, then $\alpha'$ is $\theta'$-stable (see Theorem~\ref{thm:main}).   
In good circumstances, with an iterated use of the above transformations we can get to 
a quiver-dimension vector pair that is "smaller" in an appropriate sense than $(\qui,\alpha)$ 
(see Corollary~\ref{cor:main}, Corollary~\ref{cor:sketch}).  

These reduction steps are higher dimensional analogues of the transformations  
from \cite{domokos-joo} that were sufficient to prove that in each dimension there are only finitely many toric varieties that can be realized as a moduli space of thin representations of quivers. The reduction steps in \cite{domokos-joo} were proved by establishing integral-affine equivalences between lattice polytopes. The higher dimensional analogues in the present paper are proved with the aid of 
classical invariant theory (and in the case of the result about stability, some basic facts on the geometry of representation spaces of quivers). 

This technique (combined with a combinatorial study of dimension vectors in the fundamental set, see Theorem~\ref{thm:bounding the quiver}) suffices to prove 
Theorem~\ref{thm:fundamental domain}, asserting that in each dimension only finitely many algebraic varieties can be realized as a moduli space of the form 
$\moduli(\qui,\alpha,\theta)$, where $\alpha$ is a $\theta$-stable dimension vector in the fundamental set for $\qui$. Note that to classify moduli spaces of quiver representations it is sufficient to deal with stable dimension vectors (see Proposition~\ref{prop:product}). 
A $\theta$-stable dimension vector is necessarily a root. 
Non-trivial moduli spaces  can occur only for a non-real root, and every non-real root belongs to the Weyl group orbit of a root in the fundamental set (see \cite{kac}).  
Therefore it seems natural to focus first on the dimension vectors  in the fundamental set, 
as we do in Theorem~\ref{thm:fundamental domain}. 

As an application of this result, we show in Proposition~\ref{prop:2 or n1} that under certain restrictions on the dimension vector all moduli spaces are isomorphic to products of the moduli spaces covered by Theorem~\ref{thm:fundamental domain}. 
The assumption on the dimension vector here is very restrictive, but still this 
gives Corollary~\ref{cor:toric generalization}, which is 
a direct generalization of the corresponding statement from \cite{domokos-joo} on the toric case.

In Section~\ref{sec:prel} we recall notation and terminology related to quivers, their representation spaces and the moduli spaces of quiver representations.  
Our main results are stated in Section~\ref{sec:main results}, and the steps in their proofs are outlined. In Section~\ref{sec:invariant theory} we develop our reduction steps that provide isomorphisms between algebras of relative invariants of different  quiver-dimension vector pairs. In Section~\ref{sec:stable} we show that these reduction steps are  compatible with stability.  
In Section~\ref{sec:fundamental set} we study dimension vectors in the fundamental set of a wild quiver.  
The results proved in Section~\ref{sec:invariant theory}, Section~\ref{sec:stable}, and 
Section~\ref{sec:fundamental set} have some independent interest, as they may be 
useful in further studies of moduli spaces of representations of quivers. 
We use them in  Section~\ref{sec:moduli} and Section~\ref{sec:consequences} 
to prove our finiteness statements (Theorem~\ref{thm:fundamental domain}, 
Corollary~\ref{cor:toric generalization}).    
 For the latter we need to deduce from the literature 
in Section~\ref{sec:products} the proof 
of Proposition~\ref{prop:product} on the product decomposition  of 
a moduli space of quiver representations coming from the $\theta$-stable decomposition. 
In Section~\ref{sec:example} we provide an example 
indicating that in order to prove that there are finitely many moduli spaces of quiver representations in each dimension, one probably needs to develop further techniques.  
Finally, in Section~\ref{sec:affine} we comment on the affine counterpart of these moduli spaces, the varieties parameterizing semisimple representations  of quivers. 


\section{Preliminaries and main statements} \label{sec:prel} 

\subsection{Quivers and the associated quadratic forms} 

A \emph{quiver} $\qui$ is a finite directed graph with (non-empty) vertex set $\qui_0$ and arrow set 
$\qui_1$; for an arrow $a\in \qui_1$ denote by $\source a\in \qui_0$ its source and by $\target a\in \qui_0$ its target. 
A \emph{path} in $\qui$ is a sequence $a_1,\dots,a_n$ of arows such that $\target a_i=\source a_{i+1}$ 
for $i=1,\dots,n-1$. A quiver $\qui$ is  \emph{acyclic} if there is no path in $\qui$ of positive length starting and ending at the same vertex. 
By the \emph{underlying graph} of $\qui$ we mean the graph (possibly with multiple edges) obtained by forgetting the orientation of the arrows, 
and we shall keep the symbol $\ugqui$ to denote it. So the vertex set of this graph is 
$\qui_0$ and the edge set of this graph is $\ugqui_1$, where we forget the orientation of $a\in \qui_1$ when we consider it as an edge of $\ugqui$. 
We say that the quiver $\qui$ is \emph{connected} if its underlying graph 
is connected. The \emph{degree} of a vertex $v$ of $\ugqui$ is 
\[\deg_{\ugqui}(v)=|\{a\in \qui_1\mid \source a=v\}|
+|\{a\in \qui_1\mid \target a=v\}|.\]  
A \emph{dimension vector} is an element $\alpha\in \mathbb{N}^{\qui_0}$ (interpreted as a map from $\qui_0$ into the set $\mathbb{N}$  of non-negative integers). 
Given a dimension vector $\alpha$ write $\mathrm{supp}(\alpha)$ for the full subquiver of $\qui$ spanned by the vertices $v\in \qui_0$ with $\alpha(v)\neq 0$. 
A dimension vector $\alpha$ (or the pair $(\qui,\alpha)$) is said to be \emph{sincere} if $\alpha(v)>0$ for all 
$v\in \qui_0$ (i.e. $\mathrm{supp}(\alpha)=\qui$). 
For $v\in \ugqui_0$ we shall denote by $\varepsilon_v\in \mathbb{N}^{\ugqui_0}$ the dimension vector with 
\[\varepsilon_v(w)=\begin{cases} 1\text{ for }w=v \\ 
0 \text{ for }w\neq v.\end{cases}\] 
The \emph{Ringel bilinear form} on $\mathbb{R}^{\qui_0}$ is given by 
\[\langle \alpha,\beta\rangle _{\qui}:=\sum_{v\in \qui_0}\alpha(v)\beta(v)-\sum_{a\in \qui_1}
\alpha(\source a)\beta(\target a)\ \text{ for }\alpha,\beta\in \mathbb{R}^{\qui_0}.\] 
The specialization $\alpha=\beta$ in the above formula gives the \emph{Tits quadratic form}  
\[\alpha\mapsto \langle\alpha,\alpha\rangle_{\qui}, \quad \mathbb{R}^{\ugqui_0}\to 
\mathbb{R};\] 
it depends only on the underlying graph of the quiver $\qui$, and does not depend on the orientation.  
Polarizing the Tits form we get the symmetric bilinear form 
\begin{align*}(\alpha,\beta)_{\ugqui}&:=\langle\alpha+\beta,\alpha+\beta\rangle_{\qui}-\langle\alpha,\alpha\rangle_{\qui}-\langle \beta,\beta\rangle_{\qui}=\langle\alpha,\beta\rangle _{\qui}+\langle \beta,\alpha\rangle _{\qui}
\\ & =2\sum_{v\in \ugqui_0}\alpha(v)\beta(v)-
\sum_{a\in \qui_1}(\alpha(\source a)\beta(\target a)+\alpha(\target a)\beta(\source a)).  
\end{align*}
We call it the \emph{Cartan form} associated with $\ugqui$, 
because the matrix $C_{\ugqui}$ of this bilinear form with respect to the basis 
$\{\varepsilon_v\mid v\in \ugqui_0\}$ in $\mathbb{R}^{\ugqui_0}$ is the so-called 
\emph{Cartan matrix} corresponding to $\ugqui$ (cf. \cite{kac}). 

In this paper by \emph{(extended) Dynkin graph} we mean a (extended) Dynkin graph of ADE type (listed for example in \cite[page 113]{kraft-riedtmann}). 
A quiver is called (extended) Dynkin if its underlying graph is (extended) Dynkin. 
Following \cite{kraft-riedtmann} we call 
\[F_{\ugqui}:=\{\alpha\in \mathbb{N}^{\ugqui_0}\setminus \{\underline 0\}\mid \forall v\in \ugqui_0:(\alpha,\varepsilon_v)_{\ugqui}\le 0, \ \mathrm{supp}(\alpha)\text{ is connected}\}\] 
the \emph{fundamental set} for $\qui$ (it depends only on the underlying graph of $\qui$). 
Note that $F_{\ugqui}$ is empty for a Dynkin graph $\ugqui$, 
$F_{\ugqui}$ consists of the positive integer multiples of a non-zero dimension vector $\delta$ satisfying $(\delta,\varepsilon_v)_{\ugqui}=0$ for all $v\in \ugqui_0$ for an extended Dynkin graph $\ugqui$. 
If $\qui$ is a connected quiver which is neither Dynkin nor extended 
Dynkin (i.e. $\qui$ is a connected \emph{wild quiver}), then $F_{\ugqui}$ contains an element $\alpha$ such that $(\alpha,\varepsilon_v)<0$ for some $v\in \ugqui_0$ 
by \cite[Theorem 3 and Lemma 13]{vinberg} 
(see also \cite[Lemma 1.2]{kac} or \cite[Lemma 3.1.1]{kraft-riedtmann}).

For dimension vectors $\alpha,\beta\in \mathbb{N}^{\ugqui_0}$ we write 
$\beta\le \alpha$ if $\beta(v)\le \alpha(v)$ for all $v\in \ugqui_0$, and set 
$|\alpha|:=\sum_{v\in\ugqui_0}\alpha(v)$.

\subsection{Quiver representations} \label{subsec:quivrep}

Throughout this paper we fix an algebraically closed base field $\F$ of arbitrary characteristic. 
Recall that a \emph{representation} $R$  of $\qui$ consists of a set of distinct $\F$-vector spaces 
$\{R_v\mid v \in \qui_0\}$ and an assignment $a\mapsto R_a$ $(a\in \qui_1)$, where 
$R_a$ is an $\F$-linear map from $R_{\source a}$ to $R_{\target a}$. 
We call $\underline{\dim}(R):=(\dim_{\F}(R_v)\mid v\in \qui_0)\in\mathbb{N}^{\qui_0}$ the \emph{dimension vector of} $R$. 
A representation $S$ of $\qui$ is a \emph{subrepresentation of} $R$ if $S_v$ is a subspace of $R_v$ for each $v\in \qui_0$,  
$R_a$ maps $S_{\source a}$ into $S_{\target a}$ for all $a\in \qui_1$, 
and $S_a$ is the restriction of $R_a$ to $S_{\source a}$. 

By a \emph{weight for} $\qui$ we mean an integer valued function 
$\theta\in \mathbb{Z}^{\qui_0}$ on the set of vertices of $\qui$. 
Following \cite{king}, a representation $R$ of $\qui$ is said to be 
\emph{$\theta$-semistable} for some weight $\theta\in \mathbb{Z}^{\qui_0}$ if $\sum_{v\in \qui_0}\theta(v)\dim_{\F}(R_v)=0$ and 
$\sum_{v\in \qui_0}\theta(v)\dim_{\F}(S_v)\ge 0$ for all subrepresentations $S$ of $R$. 
A non-zero $\theta$-semistable representation $R$ of $\qui$ is \emph{$\theta$-stable} if for any non-zero proper subrepresentation $S$ of $R$ we have $\sum_{v\in \qui_0}\theta(v)\dim_{\F}(S_v)> 0$. The dimension vector $\alpha\in \mathbb{N}^{\qui_0}$ is \emph{$\theta$-semistable} 
(respectively \emph{$\theta$-stable}) if there exists a $\theta$-semistable (respectively 
$\theta$-stable) representation of $\qui$ of dimension vector $\alpha$.  

The \emph{space of representations} of the quiver $\qui$ with dimension vector $\alpha\in \mathbb{N}^{\qui_0}$ is 
\[\rep(\qui,\alpha):=\bigoplus_{a\in \qui_1}\mathrm{Hom}_{\F}(\F^{\alpha(\source a)},\F^{\alpha(\target a)}).\] 
This is a non-zero $\F$-vector space provided that $\mathrm{supp}(\alpha)$ has at least one arrow. 
The \emph{base change group} $\GL(\alpha):=\prod_{v\in \qui_0}\GL(\F^{\alpha(v)})$ acts linearly on $\rep(\qui,\alpha)$: for 
$g=(g_v\mid v\in \qui_0)$ and $x=(x_a\mid a\in \qui_1)$ we have 
\[(g\cdot x)_a=g_{\target a}\circ x_a\circ g_{\source a}^{-1}.\]  
Identify $x=(x_a\mid a\in \qui_1)\in \rep(\qui,\alpha)$ with the representation of $\qui$ which assigns a copy of $\F^{\alpha(v)}$ with $v\in \qui_0$ 
and the linear map $x_a$ from 
$x_{\source a}=\F^{\alpha(\source a)}\to x_{\target a}=\F^{\alpha(\target a)}$ with $a\in \qui_1$. 
The points $x$ and $y$ in $\rep(\qui,\alpha)$ belong to the same $\GL(\alpha)$-orbit if and only if the corresponding representations are isomorphic. 

We shall denote by $\rep(\qui,\alpha)^{\theta\tiny{\text{-sst}}}$ the set of those $x\in \rep(\qui,\alpha)$ for which the corresponding representation $x$ is $\theta$-semistable, and 
denote by $\rep(\qui,\alpha)^{\theta\tiny{\text{-st}}}$ the  set of those $x\in \rep(\qui,\alpha)$ for which the corresponding representation $x$ is $\theta$-stable. 
Clearly these are $\GL(\alpha)$-invariant subsets of $\rep(\qui,\alpha)$. 

The representation space $\rep(\qui,\alpha)$ is naturally identified with 
$\rep(\qui',\alpha')$, where  $\qui'=\mathrm{supp}(\alpha)$, and 
$\alpha'=\alpha\vert_{\qui'_0}$. For a weight $\theta\in \mathbb{Z}^{\qui_0}$ denote by $\theta'$ the restriction of $\theta$ 
to $\qui'_0$. 
Then $\rep(\qui,\alpha)^{\theta\tiny{\text{-sst}}}$ is identified with 
$\rep(\qui',\alpha')^{\theta'\tiny{\text{-sst}}}$ and 
$\rep(\qui,\alpha)^{\theta\tiny{\text{-st}}}$ is identified with 
$\rep(\qui',\alpha')^{\theta'\tiny{\text{-st}}}$. 
In particular, $\alpha$ is $\theta$-semistable 
($\theta$-stable) if and only if $\alpha'$ is  $\theta'$-semistable 
($\theta'$-stable). Thus frequently in our discussions it will cause no loss of generality 
to assume that we are dealing with a sincere dimension vector for our quiver. 
Moreover, in the sequel we shall identify quiver-dimension vector pairs that are isomorphic. 
Here we say that $(\qui,\alpha)$ \emph{is isomorphic to} $(\qui',\alpha')$ if there exists a bijection 
$\pi:\qui_0\to \qui'_0$ such that for all $(v,w)\in \qui_0\times \qui_0$ we have 
$|\{a\in \qui_1\mid \source a=v,\ \target a =w\}|=|\{a'\in \qui'_1\mid \source a'=\pi(v),\ \target a' =\pi(w)\}|$ and $\alpha'(\pi(v))=\alpha(v)$.

\subsection{Semi-invariants and moduli spaces of quiver representations} \label{subsec:prel-semiinv}

Throughout this section $\alpha$ is a sincere dimension vector for a quiver $\qui$. 
We associate with a weight $\theta\in \mathbb{Z}^{\qui_0}$ the character 
\[\chi_{\theta}:\GL(\alpha)\to \F^\times,\quad \chi_{\theta}(g)=
\prod_{v\in \qui_0}\det(g_v)^{\theta(v)}\]  
(as $\alpha$ is sincere, the map $\theta\mapsto \chi_{\theta}$ from $\mathbb{Z}^{\qui_0}$ 
onto the character group of $\GL(\alpha)$ is injective). 
Denote by $\F[\rep(\qui,\alpha)]$ the coordinate ring of $\rep(\qui,\alpha)$; this is a polynomial algebra in $\sum_{a\in \qui_1}\alpha(\target a)\cdot\alpha(\source a)$ variables. 
The \emph{space of relative invariants of weight $\theta$ on $\rep(\qui,\alpha)$} is 
\begin{align*}
\F[\rep(\qui,\alpha)]^{\GL(\alpha),\theta}
:=\{p\in \F[\rep(\qui,\alpha)] \mid 
\forall g\in \GL(\alpha), \forall x\in \rep(\qui,\alpha):
\\ p(g\cdot x)=\chi_{\theta}(g)p(x)\}.\end{align*} 
This is contained in the algebra 
$\F[\rep(\qui,\alpha)]^{\SL(\alpha)}$ of \emph{semi-invariants} of $\qui$ consisting of the polynomial functions $p$ satisfying $p(g\cdot x)=p(x)$ for all $g\in \SL(\alpha)$ and 
$x\in \rep(\qui,\alpha)$, where $\SL(\alpha)=\prod_{v\in \qui_0}\SL(\alpha(v))$ is the commutator subgroup of $\GL(\alpha)$. 
Note that spanning sets of $\F[\rep(\qui,\alpha)]^{\SL(\alpha)}$ were determined in 
\cite{derksen-weyman:0}, \cite{domokos-zubkov}, \cite{schofield-vandenbergh}.  
We set 
\[ \semi(\qui,\alpha,\theta)_n=
\begin{cases} 
\F[\rep(\qui,\alpha)]^{\GL(\alpha),n\theta} &\text{ if }\theta\neq 0 
\text{ or }n=0 
\\ 0 &\text{ if }\theta=0\text{ and }n>0.
\end{cases}\] 
The algebra  $\F[\rep(\qui,\alpha)]^{\SL(\alpha)}$ contains the subalgebra 
\[\semi(\qui,\alpha,\theta):=\bigoplus_{n=0}^\infty \semi(\qui,\alpha,\theta)_n.\]  
We shall treat it as a graded algebra,  with $\semi(\qui,\alpha,\theta)_n$ as its degree $n$ homogeneous component. 
Note that 
\[\semi(\qui,\alpha,\theta)_0=\semi(\qui,\alpha,0)=\F[\rep(\qui,\alpha)]^{\GL(\alpha)},\] 
the algebra of polynomial $\GL(\alpha)$-invariants on $\rep(\qui,\alpha)$.  
Its generators were described in \cite{lebruyn-procesi} in characteristic zero and in  \cite{donkin} in arbitrary characteristic. 

Mumford's geometric invariant theory (GIT for short) was applied in \cite[Proposition 3.1]{king} to prove that 
$\rep(\qui,\alpha)^{\theta\tiny{\text{-sst}}}$ is the complement in $\rep(\qui,\alpha)$ of the common zero locus of $\bigcup_{n=1}^\infty \semi(\qui,\alpha,\theta)_n$ for a non-zero weight $\theta$ (in particular, if the dimension vector $\alpha$ is $\theta$-semistable, then there is a non-zero relative invariant on $\rep(\qui,\alpha)$ whose weight is a positive integer multiple of $\theta$), 
and $\rep(\qui,\alpha)^{0\tiny{\text{-sst}}}=\rep(\qui,\alpha)$. It is therefore a 
Zariski open dense 
$\GL(\alpha)$-invariant subset of $\rep(\qui,\alpha)$ when $\alpha$ is $\theta$-semistable, and there is a GIT-quotient 
$\pi:\rep(\qui,\alpha)^{\theta\tiny{\text{-sst}}}\twoheadrightarrow 
\rep(\qui,\alpha)^{\theta\tiny{\text{-sst}}}/\negmedspace/\GL(\alpha)$. 
Here  $\moduli(\qui,\alpha,\theta):=\rep(\qui,\alpha)^{\theta\tiny{\text{-sst}}}/\negmedspace/\GL(\alpha)$ is the quasi-projective variety obtained as the 
projective spectrum of the graded algebra 
$\semi(\qui,\alpha,\theta)$, and is called the \emph{moduli space of $\theta$-semistable representations} of $\qui$ with dimension vector $\alpha$. 
It is a coarse moduli space for families of $\theta$-semistable $\alpha$-dimensional representations of $\qui$ up to S-equivalence (cf. \cite[Proposition 5.2]{king}, 
and see \cite{newstead} for an introduction to GIT). 
This means in particular that $x$ and $y$ in $\rep(\qui,\alpha)^{\theta\tiny{\text{-sst}}}$ 
 belong to the same fiber of 
$\pi$ if and only if the representations $x$ and $y$ have the same 
composition factors in the abelian category of $\theta$-semistable 
representations of $\qui$ (whose simple objects are the $\theta$-stable representations).  
When $\alpha$ is $\theta$-stable, $\rep(\qui,\alpha)^{\theta\tiny{\text{-st}}}$ is a 
Zariski dense 
open subset of $\rep(\qui,\alpha)$, and for $x\in  \rep(\qui,\alpha)^{\theta\tiny{\text{-st}}}$, 
the fiber $\pi^{-1}(\pi(x))$ is the 
$\GL(\alpha)$-orbit of $x$. The stabilizer of $x$ consists only of the scalar multiples of the identity element of $\GL(\alpha)$ (by Schur's Lemma). 
It follows that whenever $\alpha$ is $\theta$-stable, we have 
\begin{equation} 
\label{eq:dim of stable moduli} 
\dim \moduli(\qui,\alpha,\theta)=\dim_{\F}\rep(\qui,\alpha)-\dim \GL(\alpha)+1=
1-\langle \alpha,\alpha\rangle_{\qui}. 
\end{equation} 
Note also that when  $\qui$ is acyclic, then $\semi(\qui,\alpha,\theta)_0=\F[\rep(\qui,\alpha)]^{\GL(\alpha)}=\F$, 
hence $\moduli(\qui,\alpha,\theta)$ is a projective algebraic variety over $\F$. 

\subsection{The reductions $\tau$ and $\sigma$} \label{sec:tau sigma}

\begin{definition}\label{def:tau} 
Let $\qui$ be a quiver, $\alpha$ a sincere dimension vector for $\qui$, and $u\in \qui_0$ a vertex of $\qui$ such that $\deg_{\qui}(u)>0$ and $\qui$ has no loop at $u$ (i.e. there is no arrow $a\in \qui_1$ with $\source a=u=\target a$). 
\begin{itemize}
\item[(i)] We say that $u$ is \emph{large for $(\qui,\alpha)$} if  
\[\alpha(u)\ge \max\{\sum_{b\in \qui_1\colon \target b=u}\alpha(\source b), \ 
\sum_{c\in \qui_1\colon \source c=u}\alpha(\target c)\}.\]  
\item[(ii)] Assume that $u$ is large for $(\qui,\alpha)$. 
Then 
denote by $\tau_u(\qui,\alpha)$ the quiver-dimension vector pair  $(\tau_u\qui,\tau_u\alpha)$ obtained as follows: 
\begin{itemize} 
\item remove from $\qui$ the vertex 
$u$ and the arrows adjacent to $u$; 
\item for each pair $(b,c)\in \qui_1\times \qui_1$ with 
$\target b=u$ and $\source c=u$, put a new arrow $d$ such that $\source d=\source b$ and $\target d=\target c$; 
\item  $\tau_u\alpha:=\alpha\vert_{\qui_0\setminus \{u\}}$. 
\end{itemize} 
\end{itemize}
\end{definition} 

Note that in the special case when the large vertex $u$ is a \emph{source} (i.e. there is no arrow $a\in \qui_1$ with $\target a=u$) 
or a \emph{sink} (i.e. there is no $a\in \qui_1$ with $\source a=u$),  then 
$\tau_u(\qui,\alpha)$ is simply obtained from $(\qui,\alpha)$ by removing $u$ and all arrows adjacent to $u$.

We need to recall the \emph{reflection transformations} (in the sense of  \cite{bernstein-gelfand-ponomarev}). 

\begin{definition}\label{def:sigma} 
Let $\qui$ be a quiver, $\alpha$ a sincere dimension vector for $\qui$,
and $u\in \qui_0$ a source (respectively  a sink) in $\qui$. 
\begin{itemize} 
\item[(i)] We say that $u$ is a \emph{small source} (resp. \emph{small sink}) 
for $(\qui,\alpha)$ if 
\begin{equation*}
\label{eq:1/2sum}
\sum_{a\in \qui_1\colon \source a=u}\alpha(\target a)>\alpha(u)
\quad \mbox{ (resp. }\sum_{a\in \qui_1\colon \target a=u}\alpha(\source a)>\alpha(u)\mbox{)}.  
\end{equation*} 
\item[(ii)] Assume that $u$ is a small source (respectively sink) for $(\qui,\alpha)$ 
Then $\sigma_u(\qui,\alpha)$ is the quiver-dimension vector pair 
$(\sigma_u\qui,\sigma_u\alpha)$ obtained as follows: 
\begin{itemize} 
\item reverse all the arrows in $\qui$ adjacent to $u$; 
\item $(\sigma_u\alpha)(v)=\begin{cases} \alpha(v) &\quad\text{if}\quad v\neq u;\\
-\alpha(u)+\sum_{a\in \qui_1\colon \source a=u}\alpha(\target a)&\quad \text{if}\quad v=u\mbox{ is a source}; 
\\ -\alpha(u)+\sum_{a\in \qui_1\colon \target a=u}\alpha(\source a)&\quad \text{if}\quad v=u\mbox{ is a sink.}\end{cases}$     
 \end{itemize} 
\end{itemize}
\end{definition} 

\begin{definition} \label{def:tausigma-minimal} 
Given a class  $\mathcal{C}$ of sincere quiver-dimension vector pairs and $(\qui,\alpha)\in \mathcal{C}$, 
we say that $(\qui,\alpha)\in \mathcal{C}$ is 
\emph{not $\tau\sigma$-minimal in $\mathcal{C}$} if there exists a sequence 
$(\qui^{(i)},\alpha^{(i)})$, $i=0,1,\dots,n$ of sincere quiver-dimension vector pairs such that 
\begin{enumerate} 
\item $(\qui^{(0)},\alpha^{(0)})=(\qui,\alpha)$; 
\item $(\qui^{(n)},\alpha^{(n)})\in \mathcal{C}$; 
\item For each $i=0,1,\dots,n-1$ we have 
$(\qui^{(i+1)},\alpha^{(i+1)})=\tau(\qui^{(i)},\alpha^{(i)})$, where 
$\tau=\tau_{u_i}$ for some large vertex $u_i$ for $(\qui^{(i)},\alpha^{(i)})$, or 
$\tau=\sigma_{u_i}$ for some small source or sink $u_i$ for $(\qui^{(i)},\alpha^{(i)})$; 
\item $|\qui^{(n)}_0|<|\qui_0|$, or  $|\qui^{(n)}_0|=|\qui_0|$ and $|\alpha^{(n)}|<|\alpha|$. 
\end{enumerate} 
Otherwise we say that $(\qui,\alpha)$ is \emph{$\tau\sigma$-minimal}. 
\end{definition} 

\begin{example} 
The picture below indicates three quiver-dimension vector pairs. 
The second is obtained from the first by applying $\sigma_u$. 
The third is obtained from the second by applying $\tau_w$. 
This shows that the first is not $\tau\sigma$-minimal, since the third quiver has one less vertices (and also the total dimension drops). Note that to get to the third, first we had to apply $\sigma_u$, which had increased the total dimension, and did not decrease the number of vertices. 
\end{example} 

\begin{center}
\begin{tikzpicture}[>=open triangle 45,scale=0.7]
\foreach \x in {(0,0.7),(0,-0.7),(1,0),(2.3,0)} \filldraw \x circle (2pt); 
\node [left] at (0,0.7) {$1$};
\node [left] at (0,-0.7) {$2$};
\node [above] at (1,0) {$5$};
\node [below] at (1,0) {$w$};
\node [below] at (0,-0.7) {$u$};
\node [above] at (2.3,0) {$4$};
\draw [->,thick ] (0,0.7) to (1,0); 
\draw [->,thick ] (1,0) to (0,-0.7); 
\draw [->,thick ] (1,0) to (2.3,0); 
\draw [->,thick ] (2.3,0) to (3.3,0.7); 
\draw [->,thick ] (2.3,0) to (3.3,0.3); 
\draw [->,thick ] (3.3,-0.7) to (2.3,0); 
\node at (3.3,0) {$\vdots$};
\end{tikzpicture}\qquad
\begin{tikzpicture}[>=open triangle 45,scale=0.7]
\foreach \x in {(0,0.7),(0,-0.7),(1,0),(2.3,0)} \filldraw \x circle (2pt); 
\node [left] at (0,0.7) {$1$};
\node [left] at (0,-0.7) {$3$};
\node [above] at (1,0) {$5$};
\node [below] at (1,0) {$w$};
\node [below] at (0,-0.7) {$u$};
\node [above] at (2.3,0) {$4$};
\draw [->,thick ] (0,0.7) to (1,0); 
\draw [->,thick ] (0,-0.7) to (1,0); 
\draw [->,thick ] (1,0) to (2.3,0); 
\draw [->,thick ] (2.3,0) to (3.3,0.7); 
\draw [->,thick ] (2.3,0) to (3.3,0.3); 
\draw [->,thick ] (3.3,-0.7) to (2.3,0); 
\node at (3.3,0) {$\vdots$};
\end{tikzpicture}\qquad 
\begin{tikzpicture}[>=open triangle 45,scale=0.7]
\foreach \x in {(0,0.7),(0,-0.7),(2,0)} \filldraw \x circle (2pt); 
\node [left] at (0,0.7) {$1$};
\node [left] at (0,-0.7) {$3$};
\node [below] at (0,-0.7) {$u$};
\node [above] at (2,0) {$4$};
\draw [->,thick ] (0,0.7) to (2,0); 
\draw [->,thick ] (0,-0.7) to (2,0); 
\draw [->,thick ] (2,0) to (3,0.7); 
\draw [->,thick ] (2,0) to (3,0.3); 
\draw [->,thick ] (3,-0.7) to (2,0); 
\node at (3,0) {$\vdots$};
\end{tikzpicture}
\end{center}

\subsection{Statements of the main results} \label{sec:main results}

\begin{theorem}\label{thm:main} 
Suppose that $(\qui',\alpha')=\tau(\qui,\alpha)$, 
where $(\qui,\alpha)$ is a sincere quiver-dimension vector pair, 
and $\tau=\tau_u$, where $u$ is a large vertex for $(\qui,\alpha)$, or $\tau=\sigma_u$, 
where $u$ is a small source or small sink for $(\qui,\alpha)$. 
Then for any weight $\theta\in \mathbb{Z}^{\qui_0}$, 
\begin{itemize}
\item[(i)] there exists a weight $\theta'\in \mathbb{Z}^{\qui_0}$ 
such that the algebraic varieties   
$\moduli(\qui,\alpha,\theta)$ and  $\moduli(\qui',\alpha',\theta')$ are isomorphic. 
\item[(ii)] if $\alpha$ is $\theta$-stable, then there exists a weight 
$\theta'\in \mathbb{Z}^{\qui'_0}$ such that  $\alpha'$ 
is $\theta'$-stable, and 
$\moduli(\qui,\alpha,\theta)$ and  $\moduli(\qui',\alpha',\theta')$ are isomorphic. 
\end{itemize}  
\end{theorem} 

Theorem~\ref{thm:main} immediately follows from  
the invariant theoretic Lemma~\ref{lemma:shrink}, Lemma~\ref{lemma:reflection}  (proved  in Section~\ref{sec:invariant theory}, where the weight $\theta'$ is given explicitly) and Lemma~\ref{lemma:stable preserved} (proved in Section~\ref{sec:stable}).  Note that part (i) for the operation $\sigma_u$ is a corollary of a statement from \cite{skowronski-weyman}.  

Let us record the following general consequences (their proof is explained in detail in Section~\ref{sec:moduli}) of Theorem~\ref{thm:main} and formula 
\eqref{eq:dim of stable moduli}, 
which will furnish the logical structure of our finiteness results on quiver moduli spaces 
with certain restrictions on the dimension vector.   

\begin{corollary}\label{cor:main} 
Let $\mathcal{C}$ be a set of sincere quiver-dimension vector pairs.  
Suppose that the variety $X$ can be realized as $\moduli(\qui,\alpha,\theta)$, where $(\qui,\alpha)\in \mathcal{C}$ and  $\alpha$ is $\theta$-stable.  
Then $(\qui,\alpha)$ above can be chosen to be also $\tau\sigma$-minimal in $\mathcal{C}$. 
\end{corollary} 

\begin{corollary} \label{cor:sketch} 
Let $\mathcal{C}$ be a set of sincere quiver-dimension vector pairs.  
Suppose that for some positive integer $d$, there are only finitely many 
$(\qui,\alpha)\in \mathcal{C}$ that are $\tau\sigma$-minimal in $\mathcal{C}$, 
and $1-\langle \alpha,\alpha\rangle_{\qui}=d$. 
Then up to isomorphism there are only finitely many $d$-dimensional algebraic varieties $X$ such that $X\cong \moduli(\qui,\alpha,\theta)$, where $(\qui,\alpha)\in \mathcal{C}$ 
and $\alpha$ is $\theta$-stable.  
\end{corollary} 

Theorem~\ref{thm:main} implies that  in order to classify moduli spaces of quiver representations, it is sufficient to take into account $\tau\sigma$-minimal quiver-dimension vector pairs only.  
Moreover, we need to consider only weights $\theta$ such that the dimension vector $\alpha$ is $\theta$-stable, thanks to 
Proposition~\ref{prop:product} below, which can be easily derived from results of 
\cite{chindris}, \cite{derksen-weyman}, and known facts on tame quivers 
(see Section~\ref{sec:products}).  

\begin{proposition}\label{prop:product} 
A projective algebraic variety $X$ occurs as a moduli space of representations of a quiver if and only if 
$X\cong X_1\times \cdots \times X_n$, where 
$X_i\cong \mathbb{P}^{d_i}$ (the $d_i$-dimensional projective space), or $X_i$ is isomorphic to a moduli space 
$\moduli(\qui^{(i)},\alpha^{(i)},\theta^{(i)})$ where $\qui^{(i)}$ is a connected wild acyclic 
quiver, and $\alpha^{(i)}$ is a sincere $\theta^{(i)}$-stable dimension vector for 
$\qui^{(i)}$.    
\end{proposition}

\begin{remark}\label{remark:schur root} 
Note that if $\alpha$ is a $\theta$-stable dimension vector for some quiver $\qui$, 
then the generic representation of dimension vector $\alpha$ is indecomposable, so $\alpha$ is a 
so-called \emph{Schur root}. 
\end{remark} 

In particular, if $\alpha$ is a $\theta$-stable dimension vector, then $\alpha$ is a root. 
According to the classification of roots in \cite{kac}, the set of imaginary roots for $\qui$ is the union of the Weyl group orbits of the elements of the fundamental set $F_{\ugqui}$. 
Moreover, the elements of $F_{\qui}$ for wild $\qui$ are Schur roots (see \cite[Therorem 3.3]{kraft-riedtmann}). 
Therefore it seems natural to focus first on the $\theta$-stable dimension vectors in 
$F_{\ugqui}$ as we do in Theorem~\ref{thm:fundamental domain} below 
(which deals with not necessarily projective moduli spaces and quivers that may 
have oriented cycles). 

\begin{theorem} \label{thm:fundamental domain} 
For any positive integer $d$, up to isomorphism there are only finitely many algebraic varieties $X$ of dimension at most $d$ such that there exists a quiver $\qui$, 
a weight $\theta\in \mathbb{Z}^{\qui_0}$, a $\theta$-stable dimension vector $\alpha$ in the 
fundamental set $F_{\ugqui}$ with $X\cong \moduli(\qui,\alpha,\theta)$. 
\end{theorem} 

Theorem~\ref{thm:fundamental domain} is concluded in  Section~\ref{sec:moduli} 
from  Corollary~\ref{cor:sketch}, whose assumptions are satisfied by 
Theorem~\ref{thm:bounding the quiver} (proved in 
Section~\ref{sec:fundamental set}) below: 

\begin{theorem} \label{thm:bounding the quiver} 
Let $(\qui,\alpha)$ be a $\tau\sigma$-minimal element in 
\[\mathcal{C}:=\{(\qui',\alpha') \mid \qui'\mbox{ is  a wild connected quiver, }\alpha'\in F_{\qui'},\ 
\mathrm{supp}(\alpha')=\qui'
\}.\] 
Set $d:=1-\langle \alpha,\alpha\rangle_{\qui}$.  
Then 
\begin{itemize}
\item[(i)] $|\qui_0| \le 2(d-1)+ 36(d-1)^2(12d-11)$;   
\item[(ii)] $|\qui_1|\le 2(d-1)+ 36(d-1)^2(12d-11)+12(d-1)^2$;  
\item[(iii)] $\max\{\alpha(v)\mid v\in \qui_0\} \le 18(d-1)$. 
\end{itemize}
\end{theorem}  

Denote by $\mathcal{X}$ a complete list of representatives of the isomorphism classes of 
projective algebraic varieties that are isomorphic to some 
\begin{itemize} 
\item projective space, or  
\item $\moduli(\qui,\alpha,\theta)$ where 
$\qui$ is a wild acyclic quiver,  $\alpha$ is a $\theta$-stable sincere dimension vector with $\alpha\in F_{\qui}$.  
\end{itemize} 
Note that by Theorem~\ref{thm:fundamental domain}, for each positive integer $d$, the set $\mathcal{X}$ contains finitely many varieties of dimension at most $d$. 
In Section~\ref{sec:consequences} we shall prove the following: 

\begin{proposition}\label{prop:2 or n1} 
Let $(\qui,\alpha)$ be a quiver-dimension vector pair, where $\qui$ is acyclic and $\alpha$ satisfies any of the following two assumptions: 
\goodbreak
\begin{itemize} 
\item[(i)] $\alpha(v)\le 2$ for all $v\in \qui_0$; 
\item[(ii)] There is a positive integer $n$ such that $\alpha(v)=n$ for all $v\in \qui_0$. 
\end{itemize} 
Then for any weight $\theta$, the projective algebraic variety $\moduli(\qui,\alpha,\theta)$ is isomorphic to a product of varieties from $\mathcal{X}$. 
\end{proposition} 

As an immediate consequence of 
Proposition~\ref{prop:2 or n1} and Theorem~\ref{thm:fundamental domain} we get 
the following generalization of the corresponding finiteness result from \cite{domokos-joo} on toric quiver varieties: 

\begin{corollary}\label{cor:toric generalization} 
For any positive integer $d$, 
up to isomorphism there are only finitely many projective algebraic varieties of dimension $d$ that 
are isomorphic to some moduli space $\moduli(\qui,\alpha,\theta)$,  where the dimension vector $\alpha$ satisfies condition (i) or (ii) from Proposition~\ref{prop:2 or n1}. 
\end{corollary}

\begin{remark} \begin{itemize} \item[(i)] 
By \cite[Theorem 3.1]{domokos:gmj} we know that for a quiver $\qui$, all moduli spaces $\moduli(\qui,\alpha,\theta)$ are smooth if and only if all moduli spaces $\moduli(\qui,\alpha,\theta)$ are isomorphic to products of projective or affine spaces if and only if  $\ugqui$ is the disjoint union of Dynkin or extended Dynkin graphs. 
\item[(ii)]Several families of interesting examples of quiver moduli spaces that are Fano varieties are worked out in 
\cite{franzen-reineke-sabatini}. 
\end{itemize} 
\end{remark} 

In Section~\ref{sec:example} we provide an example of a wild quiver $\qui$ and a dimension vector $\alpha$ not contained in the fundamental set $F_{\qui}$, such that 
the pair $(\qui,\alpha)$ can not be transformed by our operations of the form $\tau_u$, $\sigma_u$ to a quiver-dimension vector pair where the dimension vector is contained in the fundamental set.


\section{Some invariant theory} \label{sec:invariant theory} 

\begin{lemma}\label{lemma:shrink} 
Let $(\qui,\alpha)$ be a sincere quiver-dimension vector pair, $u\in \qui_0$ a large vertex for $(\qui,\alpha)$,  and 
let $\theta\in \mathbb{Z}^{\qui_0}$ be a weight such that $\alpha$ is $\theta$-semistable.  Then exactly one of the following holds for 
the weight $\theta$ and the dimension vector $\alpha$: 
\begin{itemize}
\item[(a)] $\theta(u)>0$ and $\alpha(u)=\sum_{b\in \qui_1\colon \target b=u}\alpha(\source b)$; 
\item[(b)] $\theta(u)<0$ and $\alpha(u)=\sum_{c\in \qui_1 \colon \source c=u}\alpha(\target c)$; 
\item[(c)] $\theta(u)=0$. 
\end{itemize}
Moreover, denote by $\tau_u\theta$  the weight for $\tau_u\qui$  given by 
\begin{itemize}
\item $(\tau_u\theta)(v)=\theta(v)+|\{b\in \qui_1\mid \target b=u,\ \source b=v\}|\cdot \theta(u)$ 
 in case (a); 
\item $(\tau_u\theta)(v)=\theta(v)+|\{c\in\qui_1\mid \source c=u,\ \target c=v\}|\cdot \theta(u)$  in case (b); 
\item $(\tau_u\theta)(v)=\theta(v)$ for all $v\in(\tau_u\qui)_0=\qui_0\setminus\{u\}$ in case (c). 
\end{itemize}
Then we have the isomorphism of graded algebras 
\[\semi(\qui,\alpha,\theta)\cong \semi(\tau_u\qui,\tau_u\alpha,\tau_u\theta)
\text{ if }\tau_u\theta\neq 0\text{ or }\theta=0,\] 
whereas 
$\semi(\qui,\alpha,\theta)$ is isomorphic as a graded algebra to the univariate polynomial algebra 
$\semi(\tau_u\qui,\tau_u\alpha,\tau_u\theta)[t]$ (where the degree of $t$ is $1$) if 
$\tau_u\theta=0$ and $\theta\neq 0$. 
\end{lemma}
\begin{proof} 
Since $u$ is assumed to be large for $(\qui,\alpha)$, 
there is no loop adjacent to $u$. 
Denote by $b_1,\dots,b_k$ the arrows of $\qui$ whose target is $u$ and denote by $c_1,\dots,c_l$ the arows of $\qui$ whose source is $u$ (so $\deg_{\ugqui}(u)=k+l$). The cases when $u$ is a source (i.e. $k=0$) or when $u$ is a sink (i.e. $l=0$) are allowed.
Denote by  $d_{i,j}$ the arrow  of $\tau_u\qui$ corresponding to the pair $(b_i,c_j)$ 
for each pair $(i,j)\in \{1,\dots,k\}\times \{1,\dots,l\}$, so  
$\source d_{i,j}=\source b_i$ and  $\target d_{i,j}=\target c_j$ . We have 
$(\tau_u\qui)_0=\qui_0\setminus \{u\}$ and 
$(\tau_u\qui)_1=\qui_1\setminus \{b_1,\dots,b_k,c_1,\dots,c_l\}\sqcup \{d_{i,j}\mid i=1,\dots,k; j=1,\dots,l\}$. Moreover, 
$\tau_u\alpha=\alpha\vert_{\qui_0\setminus \{u\}}$. 

For $x\in \rep(\qui,\alpha)$ consider the $\alpha(u)\times\alpha(\source b_i)$ matrices 
$x_{b_i}$ $(i=1,\dots,k)$. Put them in a row as the blocks of a matrix 
\[x_b:=
\left(\begin{array}{ccc}x_{b_1} & \dots & x_{b_k}\end{array}\right)
\in \F^{\alpha(u)\times \sum_{i=1}^k\alpha(\source b_i)}.\] 
Similarly, put the $\alpha(\target c_j)\times \alpha(u)$ matrices $x_{c_j}$ in a column as blocks of a 
matrix 
\[x_c:=\left(\begin{array}{c}x_{c_1} \\\vdots \\x_{c_l}\end{array}\right) 
\in \F^{(\sum_{j=1}^l\alpha(\target c_j))\times\alpha(u)}.\]
When $\alpha(u)=\sum_{i=1}^k\alpha(\source b_i)$, we denote by $\det(b)$ the polynomial function 
$\rep(\qui,\alpha)\to \F$, $x\mapsto \det(x_b)$. 
Similarly, when $\alpha(u)=\sum_{j=1}^l\alpha(\target c_j)$, we denote by $\det(c)$ the polynomial function 
$\rep(\qui,\alpha)\to \F$, $x\mapsto \det(x_c)$. 
For $g\in \GL(\alpha)$ and $x\in \rep(\qui,\alpha)$ we have 
\[(g\cdot x)_b=g_ux_b\mathrm{diag}(g_{\source b_1},\dots,g_{\source b_k})^{-1} 
\text{ and }
(g\cdot x)_c=\mathrm{diag}(g_{\target c_1},\dots,g_{\target c_l})x_cg_u^{-1}.\] 
It follows by multiplicativity of the determinant that $\det(b)$ (when defined) is a relative invariant of weight 
$\beta:=\varepsilon_u-\sum_{i=1}^k\varepsilon_{\source b_i}$ and 
$\det(c)$ (when defined) is a relative invariant of weight 
$\gamma:=-\varepsilon_u+\sum_{j=1}^l\varepsilon_{\target c_j}$. 

Identify $\GL(\tau_u\alpha)$ with the subgroup 
$\{g\in \GL(\alpha)\mid g_u=\mathrm{id}_{\F^{\alpha(u)}}\}$ of $\GL(\alpha)$ and identify 
$\GL(\F^{\alpha(u)})$ with the subgroup 
$\{g\in \GL(\alpha)\mid \forall v\neq u:\ g_v=\mathrm{id}_{\F^{\alpha(v)}}\}$. 
This way $\GL(\alpha)=\GL(\tau_u\alpha)\times \GL(\F^{\alpha(u)})$. 
We have 
\[\F[\rep(\qui,\alpha)]^{\GL(\alpha),n\theta}\subseteq 
\F[\rep(\qui,\alpha)]^{\GL(\F^{\alpha(u)}),n\theta(u)}
\subseteq \F[\rep(\qui,\alpha)]^{\SL(\F^{\alpha(u)})},\]
where  $\F[\rep(\qui,\alpha)]^{\GL(\alpha),n\theta(u)}$ stands for the space of relative 
$\GL(\F^{\alpha(u)})$-invariants on $\rep(\qui,\alpha)$ with weight $n\theta(u)$. 
Consequently, 
\[\semi(\qui,\alpha,\theta)\subseteq \F[\rep(\qui,\alpha)]^{\SL(\alpha)} \subseteq
\F[\rep(\qui,\alpha)]^{\SL(\F^{\alpha(u)})}.\] 
By the First Fundamental Theorem of Classical Invariant Theory for the special linear group (see for example \cite{weyl}, 
\cite{deconcini-procesi}), we have 
that the algebra $\F[\rep(\qui,\alpha)]^{\SL(\F^{\alpha(u)})}$ is generated by 
$\F[\rep(\qui,\alpha)]^{\GL(\F^{\alpha(u)})}$ and whichever of $\det(b)$, $\det(c)$ are  defined. 
Each function mapping $x$ to an entry of $x_c\circ x_b$ is a $\GL(\F^{\alpha(u)})$-invariant. 
When both $\det(b)$ and $\det(c)$ are defined, then 
$\det(b)\det(c)$ is the function $x\mapsto \det(x_c\circ x_b)$, thus $\det(b)\det(c)$ 
belongs to $\F[\rep(\qui,\alpha)]^{\GL(\F^{\alpha(u)})}$. 
It follows that  
\begin{equation}\label{eq:theta(u)<0}\semi(\qui,\alpha,\theta)_n\subseteq  \det(c)^{-n\theta(u)}\cdot \F[\rep(\qui,\alpha)]^{\GL(\F^{\alpha(u)})}\text{ when }\theta(u)<0 \end{equation} 
(in particular, necessarily we have $\alpha(u)=\sum_{j=1}^l\alpha(\target c_j)$ in this case), 
\[\semi(\qui,\alpha,\theta)_n\subseteq  \det(b)^{n\theta(u)}\cdot \F[\rep(\qui,\alpha)]^{\GL(\F^{\alpha(u)})}\text{ when }\theta(u)>0 \] 
(in particular, necessarily we have $\alpha(u)=\sum_{i=1}^k\alpha(b_i)$ in this case), 
and 
\[\semi(\qui,\alpha,\theta)\subseteq   \F[\rep(\qui,\alpha)]^{\GL(\F^{\alpha(u)})}\text{ when }\theta(u)=0. \] 
This shows that exactly one of (a), (b), or (c) holds for $\alpha$ and $\theta$. 

Consider the polynomial map $\mu:\rep(\qui,\alpha)\to \rep(\tau_u\qui,\tau_u\alpha)$ 
given by 
\[\mu(x)_a=\begin{cases} x_a \ &\text{ for }\ a\in (\tau_u\qui)_1\setminus \{d_{i,j}\mid i=1,\dots,k; j=1,\dots,l\}\\
x_{c_j} \circ x_{b_i} \ &\text{ for }\ a=d_{i,j}.\end{cases} 
\] 
The assumption 
$\alpha(u)\ge \max\{\sum_{i=1}^k\alpha(\source b_i),\sum_{j=1}^l\alpha(\target c_j)\}$ 
implies that $\mu$ is surjective. It 
is $\GL(\F^{\alpha(u)})$-invariant, and by the First Fundamental Theorem 
of Classical Invariant Theory for the general linear group (see for example \cite{weyl}, 
\cite{deconcini-procesi}), the comorphism $\F[\rep(\tau_u\qui,\tau_u\alpha)]\to \F[\rep(\qui,\alpha)]$ of $\mu$ 
gives an isomorphism 
$\mu^*:\F[\rep(\tau_u\qui,\tau_u\alpha)]\to  \F[\rep(\qui,\alpha)]^{\GL(\F^{\alpha(u)})}$. 
Moreover, the map $\mu$ is $\GL(\tau_u\alpha)$-equivariant, implying that 
$\mu^*$ restricts to an isomorphism of graded algebras
\begin{equation}\label{eq:mu*} 
\mbox{}\semi(\tau_u\qui,\tau_u\alpha,\tau_u\theta)\cong 
\semi(\qui,\alpha,\hat\tau_u\theta), 
\end{equation} 
where $\hat\tau_u\theta\in \mathbb{Z}^{\qui_0}$ is the weight 
for $\qui$ with $\hat\tau_u\theta(u)=0$ and $\hat\tau_u\theta(v)=\tau_u\theta(v)$ for all $v\in \qui_0\setminus\{u\}$. 
In particular, if $\theta(u)=0$, then $\theta=\hat\tau_u\theta$, and we are done with the proof in this case (i.e. in case (c)). 

So it remains to relate the graded algebras $\semi(\qui,\alpha,\hat\tau_u\theta)$ and 
$\semi(\qui,\alpha,\theta)$  when $\theta(u)\neq  0$. 
Suppose that $\theta(u)<0$, i.e. we are in case (b). 
Then $\hat\tau_u\theta=\theta+\theta(u)\cdot \gamma$ where 
$\gamma$ is the weight of the relative invariant $\det(c)$. 
Note that if for some $f$ and $h$ in $\F[\rep(\qui,\alpha)]$ we have 
$f=\det(c)^{\ell} h$, then $f$ is a relative invariant of weight $\kappa$ if and only if 
$h$ is a relative invariant of weight $\kappa-\ell \cdot \gamma$. 
Therefore we have 
\begin{align*}(\det(c)^{-n\theta(u)}\cdot \F[\rep(\qui,\alpha)]^{\GL(\F^{\alpha(u)})})
\cap \F[\rep(\qui,\alpha)]^{\GL(\alpha),n\theta}
\\ =\det(c)^{-n\theta(u)}\cdot \F[\rep(\qui,\alpha)]^{\GL(\alpha),n(\theta+\theta(u)\cdot \gamma)}\end{align*}
Thus by \eqref{eq:theta(u)<0} we have 
\begin{align}\label{eq:tauhat}
\semi(\qui,\alpha,\theta)_n
&=(\det(c)^{-n\theta(u)}\cdot \F[\rep(\qui,\alpha)]^{\GL(\F^{\alpha(u)})} )\cap \semi(\qui,\alpha,\theta)_n 
\\ \notag &= (\det(c)^{-n\theta(u)}\cdot \F[\rep(\qui,\alpha)]^{\GL(\F^{\alpha(u)})} )\cap  \F[\rep(\qui,\alpha)]^{\GL(\alpha),n\theta}
\\ \notag &=\det(c)^{-n\theta(u)}\cdot \F[\rep(\qui,\alpha)]^{\GL(\alpha),n(\theta+\theta(u)\cdot \gamma)}
\\ \notag &=\det(c)^{-n\theta(u)}\cdot \F[\rep(\qui,\alpha)]^{\GL(\alpha),n\hat\tau_u\theta}. 
\end{align} 
If $\tau_u\theta=0$, then $\hat\tau_u\theta=0$, and \eqref{eq:tauhat} shows that 
$\semi(\qui,\alpha,\theta)$ is an univariate polynomial algebra over 
$\F[\rep(\qui,\alpha)]^{\GL(\alpha)}=\semi(\qui,\alpha,\hat\tau_u\theta)$ generated by $\det(c)$, and we are done by \eqref{eq:mu*}. 
When  $\tau_u\theta\neq 0$, then $\hat\tau_u\theta\neq 0$, and \eqref{eq:tauhat} shows 
\begin{equation*}\label{eq:2tauhat}
\semi(\qui,\alpha,\theta)_n=\det(c)^{-n\theta(u)}\cdot \semi(\qui,\alpha,\hat\tau_u\theta)_n.
\end{equation*}
Denote by $\varphi_n: \semi(\qui,\alpha,\hat\tau_u\theta)_n\to \semi(\qui,\alpha,\theta)_n$ the 
$\F$-linear isomorphism $f\mapsto \det(c)^{-n\theta(u)}\cdot f$, and by 
$\varphi:\semi(\qui,\alpha,\hat\tau_u\theta)\to \semi(\qui,\alpha,\theta)$ the direct 
sum 
$\varphi=\bigoplus_{n=0}^{\infty}\varphi_n: \bigoplus_{n=0}^\infty \semi(\qui,\alpha,\hat\tau_u\theta)_n\to \bigoplus_{n=0}^\infty \semi(\qui,\alpha,\theta)_n$. 
This is an isomorphism of graded vector spaces. Moreover, $\varphi$ is an isomorphism 
of algebras as well: indeed, it is sufficient to check multiplicativity of $\varphi$ on homogeneous elements. Take $f\in \semi(\qui,\alpha,\hat\tau_u\theta)_m$ and 
$h\in \semi(\qui,\alpha,\hat\tau_u\theta)_n$. Then $fh\in \semi(\qui,\alpha,\hat\tau_u\theta)_{m+n}$, and 
therefore 
\begin{align*}\varphi(f\cdot h)&=\det(c)^{-(m+n)\theta(u)}fh
\\ &=(\det(c)^{-m\theta(u)}f)\cdot (\det(c)^{-n\theta(u)}h)=\varphi(f)\cdot \varphi(h).
\end{align*} 
So $\varphi: \semi(\qui,\alpha,\hat\tau_u\theta)\to \semi(\qui,\alpha,\theta)$ is an isomorphism of graded algebras, and we are done by \eqref{eq:mu*}. 
Case (a) (i.e. $\theta(u)>0$)   can be dealt with similarly; one just needs to replace $\det(c)$ by $\det(b)$ in the above argument. 
\end{proof}

\begin{remark} 
Reduction steps based on Classical Invariant Theory, similar in flavour to Lemma~\ref{lemma:shrink} were used before efficiently for example in 
\cite{domokos-lenzing:1} or \cite{bocklandt-smooth}. 
\end{remark} 

We need to recall a known isomorphism between algebras of semi-invariants of quivers related by a reflection transformation (in the sense of  \cite{bernstein-gelfand-ponomarev}). 
Let $u$ be a small source or a small sink for $(\qui,\alpha)$ (see Definition~\ref{def:sigma}). 
It is proved in \cite[Proposition 2.1 a)]{kac} that the algebras $\F[\rep(\qui,\alpha)]^{\SL(\alpha)}$ and 
$\F[\rep(\sigma_u\qui,\sigma_u\alpha)]^{\SL(\sigma_u\alpha)}$ are isomorphic. 
Moreover, one can deduce from the proof the following consequence,  
stated in \cite[Theorem 23]{skowronski-weyman}: 

\begin{lemma}\label{lemma:reflection} 
Let $u$ be a small source or a small sink for the quiver-dimension vector pair $(\qui,\alpha)$, and let $\theta\in\mathbb{Z}^{\qui_0}$ be a weight such that $\alpha$ is $\theta$-semistable. Define the weight $\sigma_u\theta$ as follows: 
 \[(\sigma_u\theta)(v)=\begin{cases} -\theta(u) &\text{if}\quad v=u\\
\theta(v)+|\{b\in \qui_1\mid \target b=u,\  \source b=v\}|\cdot \theta(u)& \text{if}\quad v\neq u, 
\ u \mbox{ is a sink;}
\\ \theta(v)+|\{c\in \qui_1\mid \source c=u,\  \target c=v\}|\cdot \theta(u)& \text{if}\quad v\neq u, 
\ u \mbox{ is a source;}
\end{cases}
\]
Then the graded algebras 
$\semi(\qui,\alpha,\theta)$ and  $\semi(\sigma_u\qui,\sigma_u\alpha,\sigma_u\theta)$ are isomorphic.  
\end{lemma} 

Since \cite{skowronski-weyman} does not contain the proof of Lemma~\ref{lemma:reflection}, we present the details here. The key step is to analyze an 
isomorphism between the algebras 
$\F[\mathrm{Hom}_{\F}(\F^n,\F^m)]^{\SL(\F^m)}$ 
and $\F[\mathrm{Hom}_{\F}(\F^{n-m},\F^n)]^{\SL(\F^{n-m})}$ for $n>m$. 
Identify $\mathrm{Hom}_{\F}(\F^n,\F^m)$ with the space $\F^{m\times n}$ of $m\times n$ matrices over $\F$, and $\mathrm{Hom}_{\F}(\F^{n-m},\F^n)$ with the space 
$\F^{n\times (n-m)}$. 
The coordinate ring $A:=\F[\mathrm{Hom}_{\F}(\F^n,\F^m)]$ is the polynomial algebra 
$A=\F[x_{ij}\mid i=1,\dots,m; \ j=1,\dots,n]$, where $x_{ij}$ is the $(i,j)$-coordinate function on $\F^{m\times n}$. Similarly, $B:=\F[\mathrm{Hom}_{\F}(\F^{n-m},\F^n)]$ is the polynomial algebra 
$B=\F[y_{ij}\mid i=1,\dots,n;\ j=1,\dots,n-m]$, where $y_{ij}$ maps an $n\times (n-m)$ matrix to its $(i,j)$-entry. 
The group $\GL(\F^m)\times \GL(\F^n)$ acts on $A$ by linear substitutions of the variables: 
For $(g,h)\in \GL(\F^m)\times \GL(\F^n)$ and $x_{ij}$ we have 
that $(g,h)\cdot x_{ij}$ is the $(i,j)$-entry of $g^{-1} X h$, where 
$X$ is the generic $m\times n$ matrix with $(i,j)$-entry $x_{ij}$. 
Similarly, $\GL(\F^n)\times \GL(\F^{n-m})$ acts on $B$ by linear substitutions of the variables, namely for $(g,h)\in \GL(\F^n)\times \GL(\F^{n-m})$, 
$(g,h)\cdot y_{ij}$ is the $(i,j)$-entry of $g^{-1}Yh$ where $Y$ is the generic 
$n\times (n-m)$ matrix with $(i,j)$-entry $y_{ij}$. 
For $1\le j_1<\cdots<j_m\le n$ write $X_{j_1,\dots,j_m}$ for the determinant of the $m\times m$ minor of $X$ formed by its columns of index $j_1,\dots,j_m$, 
and write $Y_{j_1,\dots,j_m}$ for the determinant of the $(n-m)\times (n-m)$ minor of $Y$ obtained by removing from $Y$ the rows of index $j_1,\dots,j_m$.  
The First Fundamental Theorem of Classical Invariant Theory for the special linear group 
asserts that 
$A^{\SL(\F^m)}$ is generated as an $\F$-algebra by $X_{j_1,\dots,j_m}$, $1\le j_1<\cdots< j_m\le n$, and 
$B^{\SL(\F^{n-m})}$ is generated as an $\F$-algebra by $Y_{j_1,\dots,j_m}$, $1\le j_1<\cdots< j_m\le n$.   
The algebra  $A^{\SL(\F^m)}$ is graded: $A^{\SL(\F^m)}=\bigoplus_{d=0}^\infty A_d$, 
where $A_d$ stands for the $\F$-subspace spanned by products of length $d$ in the generators $X_{j_1,\dots,j_m}$. Similarly, 
$B^{\SL(\F^{n-m})}=\bigoplus_{d=0}^\infty B_d$, where $B_d$ stands for the $\F$-subspace spanned by products of length $d$ in the generators $Y_{j_1,\dots,j_m}$.

\begin{proposition}\label{prop:SL} 
For $n>m$ let $A^{\SL(\F^m)}=\bigoplus_{d=0}^\infty A_d$ and 
$B^{\SL(\F^{n-m})}=\bigoplus_{d=0}^\infty B_d$ be the graded algebras as above. 
\begin{itemize} 
\item[(i)] The map $X_{j_1,\dots,j_m}\mapsto (-1)^{j_1+\cdots+j_m}
Y_{j_1,\dots,j_m}$ ($1\le j_1<\cdots<j_m\le n$) extends to an isomorphism $\varphi:A^{\SL(\F^m)}\to B^{\SL(\F^{n-m})}$ of graded $\F$-algebras. 
\item[(ii)] The isomorphism $\varphi$ is $\SL(\F^n)$-equivariant .  
\item[(iii)] As a $GL(\F^m)$-module, $A_d$ is isomorphic to a direct sum of copies of the 
$1$-dimensional  module $\det^{-d}\vert_{\GL(\F^m)}$. 
As a $GL(\F^{n-m})$-module, $B_d$ is isomorphic to a direct sum of copies of the 
$1$-dimensional  module $\det^{d}\vert_{\GL(\F^{n-m})}$. 
\item[(iv)] As a $\GL(\F^n)$-module, $A_d$ is isomorphic to $B_d\otimes \det^d\vert_{\GL(\F^n)}$. 
\end{itemize}
\end{proposition} 

\begin{proof} 
(i) The Pl\"ucker relations generate the ideal of relations among the generators $X_{j_1,\dots,j_d}$ of $A^{\SL(\F^m)}$, and formally the same relations generate the ideal of relations among the generators $Y_{j_1,\dots,j_m}$ of $B^{\SL(\F^{n-m})}$. 
This implies the existence of the isomorphism 
$\varphi$ of $\F$-algebras with 
$\varphi(X_{j_1,\dots,j_m})=(-1)^{j_1+\cdots+j_m}Y_{j_1,\dots,j_m}$. 

(ii) 
Consider the morphism from $\mathrm{Hom}_{\F}(\F^n,\F^m)=(\F^n)^*\otimes \F^m$ 
onto the subvariety $\bigwedge^m_0(\F^n)^*$ of decomposable tensors in $\bigwedge^m(\F^n)^*$
mapping an  $m\times n$ matrix 
$M$ with rows $a_1,\dots,a_m\in \F^{1\times n}=(\F^n)^*$ to 
$a_1\wedge\dots\wedge a_m\in \bigwedge^m_0(\F^n)^*$. Its comorphism gives a $\GL(\F^n)$-equivariant 
$\F$-algebra isomorphism 
$\lambda:\F[\bigwedge^m_0(\F^n)^*]\to A^{\SL(\F^m)}$. 
Denoting by $e_1,\dots,e_n$ the standard basis in $\F^n$ and by $\varepsilon_1,\dots,\varepsilon_n$ the corresponding dual basis in $(\F^n)^*$, the above isomorphism maps 
$e_{j_1}\wedge \dots\wedge e_{j_m}$ to $X_{j_1,\dots,j_m}$. 
Similarly, consider the morphism from 
$\mathrm{Hom}_{\F}(\F^{n-m},\F^n)=(\F^{n-m})^*\otimes \F^n$ mapping an $n\times (n-m)$ matrix $L$ with columns $b_1,\dots,b_{n-m}\in \F^n$ to 
$b_1\wedge\dots\wedge b_{n-m}\in \bigwedge^{n-m}_0\F^n$. 
Its comorphism is a $\GL(\F^n)$-equivariant $\F$-algebra isomorphism 
$\mu:\F[\bigwedge^{n-m}_0\F^n]\to B^{\SL(\F^{n-m})}$, mapping 
$\varepsilon_{i_1}\wedge \dots \wedge \varepsilon_{i_{n-m}}$ to $Y_{j_1,\dots,j_m}$, where 
$i_1<\cdots <i_{n-m}$ and $\{i_1,\dots,i_{n-m}\}$ is the complement of 
$\{j_1,\dots,j_m\}$ in $\{1,\dots,n\}$. 
Define the map 
$\psi:\bigwedge^m\F^n\to (\bigwedge^{n-m}\F^n)^*=\bigwedge^{n-m}(\F^n)^*$ by setting 
\[v_1\wedge \dots\wedge v_m\wedge w_1\wedge \dots\wedge w_{n-m}=\psi(v_1\wedge\dots\wedge v_m)(w_1\wedge\dots\wedge w_{n-m})e_1\wedge \dots\wedge e_n.\] 
For $g\in \GL(\F^n)$ and $v_1\wedge \dots\wedge v_m$ we have 
\[\psi(g\cdot (v_1\wedge \dots\wedge v_m))=\det(g)g\cdot (\psi(v_1\wedge \dots\wedge v_m)).\] 
In particular, $\psi$ is an $\SL(\F^n)$-module map. 
Moreover, $\psi$ is an isomorphism between the subvarieties 
$\bigwedge^m_0\F^n$ and $\bigwedge^{n-m}_0(\F^n)^*$.  
The comorphism of $\psi$ is identified with $\varphi$ (or $-\varphi$) via the isomorphisms $\lambda$ and $\mu$. Thus $\varphi$ is $\SL(\F^n)$-equivariant, and both (ii) and (iv) hold.  

(iii) follows from multiplicativity of the determinant and the definition of the action of $\GL(\F^m)$ and $\GL(\F^{n-m})$ on $\mathrm{Hom}_{\F}(\F^n,\F^m)$ and 
$\mathrm{Hom}_{\F}(\F^{n-m},\F^n)$. 
\end{proof} 

Denote by $\Z$ the center $\{(z(v)I_{\alpha(v)}\mid v\in \qui_0)\mid z:\qui_0\to \F^{\times}\}$ 
of $\GL(\alpha)$, consisting of the elements whose components corresponding to the vertices of $\qui$ are scalar transformations. The commutator subgroup of $\GL(\alpha)$ is $\SL(\alpha)$ and we have $\GL(\alpha)=\Z\cdot \SL(\alpha)$. It follows that given a $\GL(\alpha)$-module $W$, an element $w\in W$ spans a $1$-dimensional 
$\GL(\alpha)$-invariant subspace if and only if $w$ is fixed by $\SL(\alpha)$ and $\F w$ is a $\Z$-invariant subspace. 

\bigskip
 \begin{proofof}{Lemma~\ref{lemma:reflection}}
 Let $u$ be a sink vertex as in the statement. 
 Then 
 \begin{align*} \rep(\qui,\alpha)=\mathrm{Hom}_{\F}(\F^n,\F^m)\oplus \rep(\qui',\alpha') \mbox{ and } \\ 
 \rep(\sigma_u\qui,\sigma_u\alpha)=\mathrm{Hom}_{\F}(\F^{n-m},\F^n)\oplus \rep(\qui',\alpha') 
 \end{align*} 
 where $\qui'$ is the quiver obtained from $\qui$ by omitting the vertex $u$ and the arrows adjacent to $u$,  $\alpha'$ is the restriction of $\alpha$ to $\qui'_0=\qui_0\setminus \{u\}$, 
 $\F^n:=\bigoplus_{a\in \qui_1,\ \target a=u}\F^{\alpha(\source a)}$,  
 $\F^m=\F^{\alpha(u)}$, $\F^{n-m}=\F^{\sigma_u\alpha(u)}$. 
 The algebra 
 $\semi(\qui,\alpha,\theta)$ is a subalgebra of 
 $\F[\mathrm{Hom}_{\F}(\F^n,\F^m)]^{\SL(\F^m)}\otimes \F[\rep(\qui',\alpha')]$, 
whereas 
$\semi(\sigma_u\qui,\sigma_u\alpha,\sigma_u\theta)$ is a subalgebra of 
$\F[\mathrm{Hom}_{\F}(\F^{n-m},\F^n)]^{\SL(\F^{n-m})}\otimes \F[\rep(\qui',\alpha')]$. 
 The isomorphism 
 \[\varphi:\F[\mathrm{Hom}_{\F}(\F^n,\F^m)]^{\SL(\F^m)}\to
 \F[\mathrm{Hom}_{\F}(\F^{n-m},\F^n)]^{\SL(\F^{n-m})}\] 
 from Proposition~\ref{prop:SL} (i) extends to an algebra isomorphism
 \[\tilde\varphi:=\varphi\otimes \mathrm{Id}_{\F[\rep(\qui',\alpha')]}:
 \F[\rep(\qui,\alpha)]^{\SL(\F^m)}\to \F[\rep(\sigma_u\qui,\sigma_u\alpha)]^{\SL(\F^{n-m})}.\] 
 Moreover,  the isomorphism $\tilde\varphi$ is $\SL(\alpha')$-equivariant, where we 
 consider the action of $\SL(\alpha')$ on  $\F[\mathrm{Hom}_{\F}(\F^n,\F^m)]$ 
 and on $\F[\mathrm{Hom}_{\F}(\F^{n-m},\F^n)]$ 
 induced by the action of  $\SL(\F^n)$ as follows: 
 set $\qui_0(u):=\{v\in \qui_0\mid \exists a\in \qui_1: \target a=u,\ \source a=v\}$, 
 and consider the group embedding 
 \begin{align*}\iota:\prod_{v\in \qui_0(u)}\SL(\F^{\alpha(v)})\hookrightarrow 
 \SL(\oplus_{a\in\qui_1:\target a=u}\F^{\alpha(\source a)})=\SL(\F^n), 
 \\ 
  (g_v\mid v\in \qui_0(u))\mapsto (g_{\source a}\mid a\in \qui_1,\target a=u).
  \end{align*} 
 The composition of the projection 
 $\SL(\alpha')\twoheadrightarrow \prod_{v\in \qui_0(u)}\SL(\F^{\alpha(v)})$ and $\iota$ 
 gives a homomorphism $\SL(\alpha')\to \SL(\F^n)$, and we lift to $\SL(\alpha')$ the action 
 of $\SL(\F^n)$ on $\F[\mathrm{Hom}_{\F}(\F^n,\F^m)]$ and on $\F[\mathrm{Hom}_{\F}(\F^{n-m},\F^n)]$. 
 
So $\tilde\varphi:\F[\rep(\qui,\alpha)]^{\SL(\F^m)}\to \F[\rep(\sigma_u\qui,\sigma_u\alpha)]^{\SL(\F^{n-m})}$ is an $\SL(\alpha')$-equivariant algebra isomorphism.  
Let $f$ be a relative $\GL(\alpha)$-invariant on $\rep(\qui,\alpha)$ of weight $\theta$. 
 Then $f$ is contained in $\F[\mathrm{Hom}_{\F}(\F^n,\F^m)]^{\SL(\F^m)}\otimes \F[\rep(\qui',\alpha')]$, hence $f$ is a linear combination of products of the form 
$MX_{J_1}\cdots X_{J_t}$ 
where $M\in \F[\rep(\qui',\alpha')]$, $J_1,\dots,J_t$ are strictly increasing sequences of length $m$ of integers between $1,\dots,n$, and $t=\theta(u)$  (see Proposition~\ref{prop:SL} (iii)). 
Moreover, since $Z$ is a torus, $\F[\rep(\qui',\alpha')]$ is spanned over $\F$ by relative $Z$-invariants, hence 
we may assume that the elements $M$ above are relative $Z$-invariants. 
Note that  the elements $X_{J_i}$ are also relative $Z$-invariants. 
It follows that $f$ is a linear combination of relative $Z$-invariants of weight 
$\theta\vert_Z$ of the form 
$MX_{J_1}\cdots X_{J_{\theta(u)}}$ 
(here $\theta\vert_Z$ stands for the restriction to $Z$  of the character 
$\chi_{\theta}$ of $\GL(\alpha)$; see Section~\ref{subsec:prel-semiinv}). 
We have 
\[\tilde\varphi(MX_{J_1}\cdots X_{J_t})=MY_{J_1}\cdots Y_{J_t}.\] 
For $J=(j_1,\dots,j_m)$ with $1\le j_1<\cdots<j_m\le n$, and 
$z=(z(v)I_{\alpha(v)}\mid v\in \qui_0)\in Z$  we have 
\begin{equation}\label{eq:zX_J}
\frac{z\cdot X_J}{X_J}=(\prod_{a\in \qui_1,\ \target a=u}z(\source a)^{\alpha(\source a)})\frac{z\cdot Y_J}{Y_J}.
\end{equation}
Indeed, take a vertex $w$ connected to $u$ in $\qui$, and suppose that $b_1,\dots,b_\ell$ are the arrows in $\qui$ with source $w$ and target $u$. 
For ease of notation suppose that the entries in the first $\ell\alpha(w)$ columns of the generic 
matrix $X$ form a basis in the dual space of $\mathrm{Hom}_{\F}(\oplus_{i=1}^\ell\F^{\alpha(\source b_i)},\F^{\alpha(u)})$. 
Take an element $z$ of $Z$ whose component associated with $w$ is $cI_{\alpha(w)}$ and all other components are equal to the identity matrix of the appropriate size. 
Then $z\cdot X_J=c^rX_J$, where $r$ is the number of elements in $\{1,\dots,\ell\alpha(w)\}$ that belong to $\{j_1,\dots,j_m\}$. Clearly, the number of elements in $\{1,\dots,\ell\alpha(w)\}$ that belong to the complement in $\{1,\dots,n\}$ of $\{j_1,\dots,j_m\}$ is 
$\ell\alpha(w)-r$, and $z\cdot Y_J=c^{r-\ell\alpha(w)}Y_J$. 
This holds for all vertices $w$ that are connected to $u$. 
This explains \eqref{eq:zX_J}, which 
implies in turn that 
$\tilde\varphi(f)$ is a relative $Z$-invariant of weight $\sigma_u\theta\vert_Z$. 
On the other hand, $\tilde\varphi(f)$ is $\SL(\sigma_u\alpha)$-invariant by 
Proposition~\ref{prop:SL} (ii), and hence $\tilde\varphi(f)$ is a relative $\GL(\sigma_u\alpha)$-invariant of weight $\sigma_u\theta$. Essentially the same considerations show also that the inverse of $\tilde\varphi$ maps a relative $\GL(\sigma_u\alpha)$-invariant of weight 
$\sigma_u\theta$ to a relative $\GL(\alpha)$-invariant of weight $\theta$, and that $\tilde\varphi$ gives an isomorphism between the corresponding spaces of relative invariants. 

The case when $u$ is a source vertex is handled by an obvious modification of the above argument. 
 \end{proofof} 


\section{Our reductions preserve the stability of a dimension vector}\label{sec:stable}

\begin{lemma}\label{lemma:schofield} 
Let $\alpha$ be a $\theta$-semistable dimension vector for $\qui$. If $\alpha$ is not $\theta$-stable, 
then there exists a non-zero dimension vector $\beta\le \alpha$, $\beta\neq \alpha$ such that 
$\sum_{v\in \qui_0}\theta(v)\beta(v)=0$ and 
each $x\in \rep(\qui,\alpha)$ has a subrepresentation of dimension vector $\beta$. 
\end{lemma} 

\begin{proof} 
For a dimension vector $\beta\le \alpha$ set 
\[Z_{\beta}:=\{x\in \rep(\qui,\alpha)\mid x \mbox{ has a subrepresentation of dimension vector }
\beta\}.\] 
 Set $B:=\{\beta\in \mathbb{N}^{\qui_0}\mid \beta\le \alpha,\ \beta\neq 0, \ 
 \beta\neq \alpha,\  
\sum_{v\in \qui_0}\theta(v)\beta(v)\le0\}$. 
The assumption that $\alpha$ is $\theta$-semistable but not $\theta$-stable means that 
$\rep(\qui,\alpha)=\bigcup_{\beta\in B}Z_{\beta}$. 
On the other hand, for each $\beta\in B$ the subset $Z_{\beta}$ is Zariski closed in $\rep(\qui,\alpha)$ by 
\cite[Lemma 3.1]{schofield} (see also \cite{crawley-boevey}). 
Thus $\rep(\qui,\alpha)=Z_{\beta}$ for some 
$\beta\in B$. Since $\alpha$ is $\theta$-semistable, we have 
$\sum_{v\in \qui_0}\theta(v)\beta(v)=0$ for this $\beta$. 
\end{proof} 

\begin{lemma} \label{lemma:stable preserved} 
Let $\qui$ be a quiver, $\alpha$ a $\theta$-semistable sincere dimension vector,   
and $u$ a large vertex or a small source or sink for $(\qui,\alpha)$.  
Write 
$\tau$ for the corresponding operation $\tau_u$ or $\sigma_u$. 
Then $\tau\alpha$ is $\tau\theta$-semistable for $\tau\qui$, 
where $\tau\theta$ is the weight given in  Lemma~\ref{lemma:shrink} 
or Lemma~\ref{lemma:reflection}.  
Moreover, if $\alpha$ is $\theta$-stable for $\qui$, then 
$\tau\alpha$ is $\tau\theta$-stable for $\tau\qui$.  
\end{lemma} 

\begin{proof} 
Lemma~\ref{lemma:shrink} and Lemma~\ref{lemma:reflection} immediately imply 
that if $\semi(\qui,\alpha,\theta)$ has a non-zero homogeneous element of positive degree, 
then the same holds for $\semi(\tau\qui,\tau\alpha,\tau\theta)$, except possibly when $\tau\theta=0$, 
when all dimension vectors are semistable. 
If $\theta=0$, then $\tau\theta=0$ as well, so all dimension vectors are semistable. 
Thus the $\tau\theta$-semistability of $\tau\alpha$ folows. 

Suppose for contradiction that $\alpha$ is $\theta$-stable, but $\tau\alpha$ is not 
$\tau\theta$-stable.  By Lemma~\ref{lemma:schofield} we may choose a dimension vector 
$\beta\in \mathbb{N}^{\qui_0}$ such that 
\begin{equation}\label{eq:beta<alpha}
\beta\le \tau\alpha, \quad 0\neq \beta \neq \tau\alpha, \quad 
\sum_{v\in\tau \qui_0}\tau\theta(v)\beta(v)=0,
\end{equation} 
and all $\tau\alpha$-dimensional representation of $\tau \qui $ have a 
subrepresentation of dimension vector $\beta$.  

\medskip
\noindent Case 1: $\tau=\sigma_u$.  Assume first  that  $u$ is a small sink vertex in $\qui$. 
Since $\alpha$ is $\theta$-stable for $\qui$ and every representation of $\qui$ of dimension vector $\alpha$ has a subrepresentation of dimension vector $\varepsilon_u$, we have $\theta(u)> 0$. 
Take any representation $R$ of $\qui$ of dimension vector $\alpha$ such that 
$\bigoplus_{\target a=u}R_a:\bigoplus_{\target a=u}R_{\source a}\to R_u$ is surjective (such a representation exists by our assumption that $u$ is a small sink). Write $K$ for the kernel of  $\bigoplus_{\target a=u}R_a$. 
Denote by $R'$ the representation of $\sigma_u\qui$ such that  $R'_u=K$, 
$R'_v=R_v$ for $v\neq u$, $R'_a=R_a$ for the arrows $a$ not adjacent to $u$, 
whereas $\bigoplus_{a\in \sigma_u\qui_1,\ \source a=u}R'_a$ is the embedding of $K=R'_u$ into 
$\bigoplus_{a\in \sigma_u\qui_1, \ \source a=u}R'_{\target a}
=\bigoplus_{a\in \qui_1,\ \target a=u}R_{\source a}$. 
The surjectivity of $\bigoplus_{\target a=u}R_a:\bigoplus_{\target a=u}R_{\source a}\to R_u$ implies 
that $\underline{\dim}(R')=\sigma_u\alpha$. 
By choice of $\beta$, $R'$ has a subrepresentation 
$S'$ of dimension vector $\beta$. 
Set $S_v:=S'_v$ for $u\neq v\in \qui_0$ and 
set $S_u:=\sum_{a\in \qui_1,\ \target a=u}R_a(S'_{\source a})$.  
Now $S$ is a subrepresentation of $R$ by construction. 
Since $\theta(u)\neq 0$, \eqref{eq:beta<alpha} implies that there exists a $w\in \qui_0\setminus\{u\}$ with $\beta(w)\neq 0$ and there exists a 
$w\in \qui_0\setminus\{u\}$ with $\beta(w)<\sigma_u\alpha(w)=\alpha(w)$. 
So $S$ is a proper non-zero subrepresentation of $R$. 
We have 
\[(\bigoplus_{a\in \sigma_u\qui_1,\ \source a=u}R'_a)(S'_u)\subseteq 
\bigoplus_{a\in \sigma_u\qui_1,\ \source a=u}S'_{\target a}=\bigoplus_{a\in \qui_1,\ \target a=u}S_{\source a}\] 
since $S'$ is a subrepresentation of $R'$. 
On the other hand, 
\[(\bigoplus_{a\in \sigma_u\qui_1,\ \source a=u}R'_a)(S'_u)\subseteq 
\ker(\bigoplus_{a\in \qui_1,\ \target a=u}R_a)\]
by definition of $R'$.  
Consequently, 
\[S'_u\cong (\bigoplus_{a\in \sigma_u\qui_1,\ \source a=u}R'_a)(S'_u)
\subseteq \ker((\bigoplus_{a\in \qui_1,\ \target a=u}R_a)\big\vert_{\bigoplus_{a\in \qui_1,\target a=u}S_{\source a}}),\]
hence 
\[\beta(u)\le\dim_\F(\ker((\bigoplus_{a\in \qui_1,\ \target a=u}R_a)\big\vert_{\bigoplus_{a\in \qui_1,\target a=u}S_{\source a}}))\] 
It follows that 
\begin{align*}\dim_{\F}(S_u)=\sum_{a\in \qui_1,\ \target a=u}\dim_{\F}(S'_{\source a})-
\dim_\F(\ker((\bigoplus_{a\in \qui_1,\ \target a=u}R_a)\big\vert_{\bigoplus_{a\in \qui_1,\target a=u}S'_{\source a}}))
\\ =\sum_{a\in \qui_1,\ \target a=u}\beta(\source a)-
\dim_\F(\ker((\bigoplus_{a\in \qui_1,\ \target a=u}R_a)\big\vert_{\bigoplus_{a\in \qui_1,\target a=u}S_{\source a}}))
\\ \le -\beta(u)+\sum_{a\in \qui_1,\ \target a=u}\beta(\source a). 
\end{align*}
By $\theta(u)> 0$ we conclude 
\[\theta(u)\dim_{\F}(S_u)\le -\theta(u)\beta(u)+\theta(u)\sum_{a\in \qui_1,\ \target a=u}\beta(\source a)
. \] 
The above inequality and $\underline{\dim}(S)(v)=\beta(v)=\underline{\dim}(S')(v)$ for 
$v\neq u$ imply that 
\begin{align*} 
\sum_{v\in \qui_0}\theta(v)&\underline{\dim}(S)(v)= 
\theta(u)\dim_{\F}(S_u)+\sum_{v\in \qui_0\setminus \{u\}}\theta(v)\underline{\dim}(S)(v)
\\ &\le 
-\theta(u)\beta(u)+\theta(u)\sum_{a\in \qui_1,\ \target a=u}\beta(\source a)
+\sum_{v\in \qui_0\setminus \{u\}}\theta(v)\beta(v)
\\ &=-\theta(u)\beta(u)+
\sum_{v\in \qui_0\setminus \{u\}}
(\theta(v)+\theta(u)\cdot |\{a\in \qui_1\mid \target a=u,\ \source a=v\}|)\beta(v)
\\ &=\sum_{v\in \sigma_u\qui_0}\sigma_u\theta(v)\beta(v)=0. 
\end{align*} 
Thus $R$ has a proper non-zero subrepresentation $S$ with 
$\sum_{v\in \qui_0}\theta(v)
\underline{\dim}(S)(v)\le 0$, and consequently, $R$ is not $\theta$-stable.  
This holds for any representation $R$ of $\qui$ with dimension vector $\alpha$ for which 
$\bigoplus_{\target a=u}R_a:\bigoplus_{\target a=u}R_{\source a}\to R_u$ is surjective. 
Note that this condition holds for the representations contained in an open Zariski dense subset of $\rep(\qui,\alpha)$, so the subset of non-$\theta$-stable points in $\rep(\qui,\alpha)$ contains a Zariski open dense subset. This is a contradiction, since for the $\theta$-stable dimension vector $\alpha$ the $\theta$-stable points in $\rep(\qui,\alpha)$ constitute a Zariski dense open subset as well, and two open Zariski dense subsets of $\rep(\qui,\alpha)$ have non-empty intersection.  
The case when $\tau=\sigma_u$ and $u$ is a source vertex in $\qui$ can be dealt with by a 
similar (dual) argument. 

\medskip
\noindent Case 2: $\tau=\tau_u$. We use the notation of the proof of Lemma~\ref{lemma:shrink}. 
Suppose first that we are in case (a); that is, $\theta(u)>0$ and $\alpha(u)=\sum_{i=1}^k\alpha(\source b_i)$. 
The representation space $\rep(\qui,\alpha)$ contains the Zariski open dense subset 
\[U:=\{x\in \rep(\qui,\alpha)\mid \bigoplus_{i=1}^kx_{b_i}:\bigoplus_{i=1}^kx_{\source b_i}\to x_u 
\mbox{ is an isomorphism}\}.\] 
Denote by $\mu:\rep(\qui,\alpha)\to \rep(\tau\qui,\tau\alpha)$ the surjective map given in 
the proof of Lemma~\ref{lemma:shrink}. 
Pick some $x\in U$. 
By choice of $\beta$, $\mu(x)$ has a subrepresentation $S'$ with $\underline{\dim}(S')=\beta$. 
Set $S_v:=S'_v$ for 
$v\in \qui_0\setminus\{u\}$ and $S_u:=\sum_{i=1}^kx_{b_i}(S_{\source b_i})$. 
Then $S$ is  a subrepresentation of $x$, since 
\[x_{c_j}(x_{b_i}(S_{\source b_i}))=
\mu(x)_{d_{i,j}}(S'_{\source d_{i,j}})\subseteq S'_{\target d_{i,j}}=S_{\target c_j}.\]  
As $x$ belongs to $U$, the maps $x_{b_i}:x_{\source b_i}\to x_u$ are injective and  
\begin{align}\label{eq:dim(Su)} 
\dim_{\F}(S_u)&=\dim_{\F}(\sum_{i=1}^kx_{b_i}(S_{\source b_i}))
=\dim_{\F}((\bigoplus_{i=1}^kx_{b_i})(\bigoplus_{i=1}^kS_{\source b_i}))
\\ \notag &=\sum_{i=1}^k\dim_{\F}(S_{\source b_i})
=\sum_{i=1}^k\dim_{\F}(S'_{\source b_i})
=\sum_{i=1}^k  \beta(\source b_i), 
\\ \notag \mbox{whereas } & \dim_{\F}(S_v)=\dim_{\F}(S'_v)=\beta(v).
\end{align} 
We have 
\begin{align*} 
0&=\sum_{v\in \tau_u\qui_0}\tau_u\theta(v)\beta(v)
\\ &=\sum_{v\in \qui_0\setminus\{u\}}(\theta(v)+
\theta(u)\cdot |\{b\in \qui_1\mid \target b=u,\ \source b=v\}|)\beta(v)
\\ &=\sum_{v\in \qui_0\setminus\{u\}}\theta(v)\beta(v)+
\theta(u)\cdot \sum_{v\in \qui_0\setminus\{u\}} |\{b\in \qui_1\mid \target b=u,\ \source b=v\}|\beta(v)
\\ &= \theta(u)\sum_{i=1}^k\beta(\source b_i)+\sum_{v\in \qui_0\setminus\{u\}}\theta(v)\beta(v)
\\ &=\sum_{v\in \qui_0}\theta(v)\underline{\dim}(S)(v)
\end{align*} 
(the last equality above holds by \eqref{eq:dim(Su)}).  
Note that 
$\underline{\dim}(S)\vert_{\qui_0\setminus \{u\}}
=\underline{\dim}(S')=\beta$, 
so $S$ is a proper non-zero subrepresentation of 
$R$ with 
$\sum_{v\in \qui_0}\theta(v)\underline{\dim}(S)(v)=0$. 
It follows that no $x\in U$ is $\theta$-stable. 
Since $U$ is Zariski open dense in $\rep(\qui,\alpha)$, this contradicts the assumption that $\alpha$ is a $\theta$-stable dimension vector. 
The argument for the case (b) of Lemma~\ref{lemma:shrink} is essentially the same. 
In case (c) of Lemma~\ref{lemma:shrink}, one obviously has that  
if $x\in \rep(\qui,\alpha)$ is $\theta$-stable, then $\mu(x)\in \rep(\tau_u\qui,\tau_u\alpha)$ is 
$\tau_u\theta$-stable. 
\end{proof}


\section{Dimension vectors in the fundamental set} \label{sec:fundamental set} 

For this section we fix a wild connected quiver  $\qui$, and we fix a sincere dimension vector
$\alpha\in\mathbb{N}^{\qui_0}$  (i.e. $\alpha(v)>0$ for all $v\in \qui_0$). We shall assume that $\alpha$ belongs to the \emph{fundamental set} $F_{\ugqui}$.  
Recall that by definition of the Cartan form we have 
\begin{equation}\label{eq:alphaepsilon}(\alpha,\varepsilon_v)_{\ugqui}=2\alpha(v)-\sum_{a\in \qui_1,\ \source a=v}\alpha(\target a)
-\sum_{a\in \qui_1, \ \target a=v}\alpha(\source a),\end{equation}  
so the assumption $\alpha\in F_{\ugqui}$ means that 
\[2\alpha(v)\le \sum_{a\in \qui_1,\ \source a=v}\alpha(\target a)
+\sum_{a\in \qui_1, \ \target a=v}\alpha(\source a) 
\quad \mbox{ for all }v\in \qui_0.\]
Set 
\[\ugqui^{\alpha,-}_0:=\{v\in \ugqui_0\mid (\alpha,\varepsilon_v)_{\ugqui}<0\}.\] 
Denote by $\ugqui^{\alpha,+}$ the full subgraph of $\ugqui$ spanned by the vertices in 
\[\ugqui^{\alpha,+}_0:=\ugqui_0\setminus \ugqui^{\alpha,-}_0=
\{v\in \ugqui_0\mid (\alpha,\varepsilon_v)_{\ugqui}=0\};\] 
that is, the vertex set of $\ugqui^{\alpha,+}$ is $\ugqui^{\alpha,+}_0$, and the edge set 
of $\ugqui^{\alpha,+}$ is 
\[\ugqui^{\alpha,+}_1:=\{a\in \qui_1\mid \source a,\target a\in \ugqui^{\alpha,+}_0\}. \] 

\begin{lemma}\label{lemma:free vertex} 
Let $\Delta$ be a  subgraph of $\ugqui$, and $v$ a vertex of $\Delta$.  
We have the inequality 
\[(\alpha,\varepsilon_v)_{\ugqui}\le (\alpha\vert_{\Delta},\varepsilon_v)_{\Delta} \] 
with equality if and only if all edges of $\ugqui$ adjacent to $v$ belong to $\Delta$.  
\end{lemma} 

\begin{proof} 
Denote by $E(v)$ the set of edges in $\ugqui$ that are connected to $v$, and are not contained in $\Delta$. 
We conclude from \eqref{eq:alphaepsilon} that 
\[(\alpha\vert_{\Delta},\varepsilon_v)_{\Delta}-(\alpha,\varepsilon_v)_{\ugqui}= 
\sum_{a\in E(v),\ \source a=v}\alpha(\target a)
+\sum_{a\in E(v),\ \target a=v}\alpha(\source a).\] 
This clearly implies the statement. 
\end{proof}

\begin{lemma}\label{lemma:dynkin} 
The graph $\ugqui^{\alpha,+}$ is the disjoint union of Dynkin graphs. 
\end{lemma}

\begin{proof} 
Let $\Delta$ be a connected component of $\ugqui^{\alpha,+}$, and 
$\beta:=\alpha\vert_{\Delta}$. By Lemma~\ref{lemma:free vertex} for each $v\in \Delta_0$ we have 
\begin{align}\label{eq:beta,epsilon} 
(\beta,\varepsilon_v)_{\Delta}\ge (\alpha,\varepsilon_v)_{\ugqui} \ge 0. 
\end{align}
It follows that the Cartan matrix $C_{\Delta}$ does not fall under case (N) in  \cite[Lemma 1.2]{kac}, hence  
$\Delta$ is a Dynkin or extended Dynkin graph 
(see for example \cite[Section 1.1]{kac}  or \cite[Section 1.3]{kraft-riedtmann}). 
Moreover, since $\qui$ is connected and wild, $\Delta$ must have a vertex $v$ connected with some vertex in $\ugqui^{\alpha,-}_0$. For such a $v$ the first inequality in \eqref{eq:beta,epsilon} is strict, so we have $(\beta,\varepsilon_v)_{\Delta}>0$. Consequently, the Cartan matrix $C_{\Delta}$ does not fall under case (Z) in \cite[Lemma 1.2]{kac}.  Thus $\Delta$ is a Dynkin graph. 
\end{proof}

We introduce some ad hoc terminology. 
A vertex of $\ugqui^{\alpha,+}$ is called \emph{tied} if it is connected by an edge in $\ugqui$ to a vertex in $\ugqui^{\alpha,-}_0$. A vertex of $\ugqui^{\alpha,+}$ is called \emph{free} if it is not connected by an edge in $\ugqui$ to a vertex in $\ugqui^{\alpha,-}_0$.

\begin{lemma}\label{lemma:bounds} 
\begin{itemize}
\item[(i)] We have $|\ugqui^{\alpha,-}_0|\le-2 \langle\alpha,\alpha\rangle_{\qui}$.  
\item[(ii)] For each $v\in \ugqui^{\alpha,-}_0$ we have $\alpha(v)\le -2 \langle\alpha,\alpha\rangle_{\qui}$.  
\item[(iii)] For each $v\in \ugqui^{\alpha,-}_0$ we have $\deg_{\ugqui}(v)\le -6 \langle\alpha,\alpha\rangle_{\qui}$.  
\item[(iv)] There are at most $12 \langle\alpha,\alpha\rangle_{\qui}^2$ tied vertices in $\ugqui^{\alpha,+}_0$. 
\item[(v)] For each tied vertex $w\in \ugqui^{\alpha,+}_0$  we have $\alpha(w)\le -6 \langle\alpha,\alpha\rangle_{\qui}$. 
\item[(vi)] If $w\in \ugqui_0$ is connected by an edge of $\ugqui$ to a vertex $v\in \ugqui^{\alpha,+}_0$, then $\alpha(w)\le 2\alpha(v)$. 
\item[(vii)] We have $|\ugqui_0|\le |\ugqui^{\alpha,+}_0|-2\langle\alpha,\alpha\rangle_{\qui}$ and 
\[|\ugqui_1|\le |\ugqui^{\alpha,+}_1|-6\langle\alpha,\alpha\rangle_{\qui}|\ugqui^{\alpha,-}_0|\le 
|\ugqui^{\alpha,+}_1|+12 \langle\alpha,\alpha\rangle_{\qui}^2.\]
\end{itemize}
\end{lemma}

\begin{proof} We have 
\begin{equation}\label{eq:alphaalpha} 
-2 \langle\alpha,\alpha\rangle_{\qui}= -(\alpha,\alpha)_{\ugqui}=\sum_{v\in \ugqui_0}-(\alpha,\varepsilon_v)_{\ugqui}\cdot \alpha(v).
\end{equation}
As $\alpha\in F_{\ugqui}$, each summand on the right hand side of \eqref{eq:alphaalpha} is non-negative, and the summands 
corresponding to $v\in \ugqui^{\alpha,-}_0$ are positive, implying that 
\[\sum_{v\in \ugqui^{\alpha,-}_0}\alpha(v)\le -2 \langle\alpha,\alpha\rangle_{\qui}.\] 
Recall that $\alpha$ is a sincere dimension vector, hence each summand on the left hand side above is at least $1$, and the number of summands is 
$|\ugqui^{\alpha,-}_0|$, hence (i) holds. 
Moreover, the summand corresponding to $v\in \ugqui^{\alpha,-}_0$ is at least 
$\alpha(v)$, so by \eqref{eq:alphaalpha} we get $\alpha(v)\le -2 \langle\alpha,\alpha\rangle_{\qui}$, thus (ii) is shown.  
By \eqref{eq:alphaalpha} for each $v\in \ugqui_0$ we have 
\begin{equation} \label{eq:2alphaalpha} 
-2 \langle\alpha,\alpha\rangle_{\qui}  \ge -(\alpha,\varepsilon_v)_{\ugqui}\cdot \alpha(v)\ge 
 -(\alpha,\varepsilon_v)_{\ugqui}. 
 \end{equation} 
 By  \eqref{eq:alphaepsilon} for any $v\in \ugqui_0$ we have 
 \begin{equation} \label{eq:2alphaepsilon} 
 -(\alpha,\varepsilon_v)_{\ugqui} 
  \ge -2\alpha(v)+\deg_{\ugqui}(v), 
  \end{equation} 
  and if $w\in \ugqui_0$ is connected  to $v$ by an edge in $\ugqui$, we also have  
\begin{equation} \label{eq:3alphaepsilon} 
-(\alpha,\varepsilon_v)_{\ugqui} 
  \ge -2\alpha(v)+\alpha(w). 
\end{equation}  
Combining (ii), \eqref{eq:2alphaalpha} and \eqref{eq:2alphaepsilon} 
for $v\in \ugqui^{\alpha,-}_0$ we get that 
\[ -2 \langle\alpha,\alpha\rangle_{\qui} \ge 4\langle\alpha,\alpha\rangle_{\qui}+\deg_{\ugqui}(v),\] 
so (iii) holds. 
Similarly, combining (ii), \eqref{eq:2alphaalpha} and \eqref{eq:3alphaepsilon} 
for $v\in \ugqui^{\alpha,-}_0$ and $w\in \ugqui_0$ connected to $v$ by an edge in $\ugqui$ 
we get that  
\[ -2 \langle\alpha,\alpha\rangle_{\qui} \ge 4\langle\alpha,\alpha\rangle_{\qui}+\alpha(w),\] 
so (v) holds. 
The number of tied vertices is obviously bounded by the number of edges adjacent to a vertex in $\ugqui^{\alpha,-}_0$, so 
statement (iv) is an immediate consequence of (i) and (iii). 
For $v\in \ugqui^{\alpha,+}_0$ we have $(\alpha,\epsilon_v)_{\ugqui}=0$, 
so \eqref{eq:3alphaepsilon} yields 
\[0=-(\alpha,\varepsilon_v)_{\ugqui} 
  \ge -2\alpha(v)+\alpha(w),\] 
thus (vi) holds. The bound for $|\ugqui_0|$ in (vii) follows from the obvious equality 
$|\ugqui_0|=|\ugqui^{\alpha,+}_0|+|\ugqui^{\alpha,-}_0|$ and (i). 
To see the bound for $|\ugqui_1|$ in (vii) note that any edge in 
$\ugqui_1\setminus \ugqui^{\alpha,+}_1$ is connected to a vertex of $\ugqui^{\alpha,-}_0$, and the number of such arrows can be bounded by (i) and (iii). 
\end{proof}

\begin{remark}\label{remark:d>1} (i) Since $\qui$ is wild, 
Lemma~\ref{lemma:bounds} (i) implies $\langle\alpha,\alpha\rangle_{\qui}<0$. 

(ii) The numbers $|\ugqui^{\alpha,+}_0|$ and $|\ugqui^{\alpha,+}_1|$ can be arbitrarily large for a fixed value of $\langle\alpha,\alpha\rangle_{\qui}$. 
\end{remark}

Note that any connected subgraph of a Dynkin graph is a Dynkin graph. 
Every Dynkin graph is a tree. The graphs $E_6$, $E_7$, $E_8$, and $D_n$ ($n\ge 4$) 
have a vertex of degree $3$, three vertices of degree $1$, and their remaining vertices have degree $2$. The graphs $A_n$ ($n\ge 2$) have two vertices of degree $1$ and $n-2$ vertices of degree $2$. Removing a degree $d\in \{2,3\}$ vertex $v$ from a Dynkin graph 
$\Delta$ it splits into $d$ connected components. Adding to a component the edge connecting it to $v$ we obtain connected subgraphs $\Delta_1,\dots,\Delta_d$ of $\Delta$, 
such that $\Delta$ is their union, and for $i\neq j$, $\Delta_i$ and $\Delta_j$ have one common vertex, namely $v$. Note that $\Delta_1,\dots,\Delta_d$ are Dynkin graphs, because any connected subgraph of a Dynkin graph is a Dynkin graph. 

The above observations imply that the tied vertices decompose $\ugqui^{\alpha,+}$ as the union 
of connected subgraphs $\Delta^{(i)}$ ($i=1,\dots,\kappa_{\ugqui,\alpha}$) satisfying that 
\begin{enumerate} 
\item each $\Delta^{(i)}$ is a Dynkin graph; 
\item each $\Delta^{(i)}$ has a vertex $v^{(i)}$  with $\deg_{\Delta^{(i)}}(v^{(i)})\le 1$ 
and $v^{(i)}$ is a tied vertex; 
\item all the vertices $v$ of $\Delta^{(i)}$ with $\deg_{\Delta^{(i)}}(v)\in \{2,3\}$ are free in 
$\ugqui^{\alpha,+}$; 
\item for $i\neq j$, the subgraphs $\Delta^{(i)}$ and $\Delta^{(j)}$ are either disjoint or they have  exactly one common vertex, and this common vertex is a tied vertex of 
$\ugqui^{\alpha,+}$. 
\end{enumerate} 

\begin{lemma}\label{lemma:number of Gamma^i} 
The number $\kappa_{\ugqui,\alpha}$ of the subgraphs $\Delta^{(i)}$  of $\ugqui^{\alpha,+}$ 
introduced above is at most  
$36 \langle\alpha,\alpha\rangle_{\qui}^2$. 
\end{lemma} 

\begin{proof}  
Each $\Delta^{(i)}$ has a tied vertex, and one tied vertex belongs to at most three 
subgraphs $\Delta^{(i)}$. Thus $\kappa_{\ugqui,\alpha}$ is at most three times the number of tied vertices. Consequently by Lemma~\ref{lemma:bounds} (iv) we get 
$\kappa_{\ugqui,\alpha}\le 3\cdot 12 \cdot \langle\alpha,\alpha\rangle_{\qui}^2$. 
\end{proof}

\begin{lemma}\label{lemma:arithmetic} 
Take an $A_n$-type subgraph 
\[\Delta:\qquad 
\begin{tikzpicture}[scale=0.6]
\node at (0,0) {$v_1$}; \node at (3,0) {$v_2$} ; \node at (6,0) {$v_3$}; \node at (7.5,0) {$\cdots$}; \node at (9,0) {$v_{n-1}$}; \node at (12,0) {$v_n$} ; 
\draw (0.3,0)--(2.6,0);  \draw (3.3,0)--(5.6,0); \draw (9.6,0)--(11.6,0);
\end{tikzpicture} \qquad (n\ge 2)
\]
of $\ugqui^{\alpha,+}$, where $v_2,\dots,v_{n-1}$ have degree $2$ in $\ugqui$.   
\begin{itemize} 
\item[(i)] Then $\alpha(v_1),\dots,\alpha(v_n)$ is an arithmetic progression. 
\item[(ii)] If $v_n$ has degree $1$ in $\ugqui$, then $\alpha(v_{n-1})=2\alpha(v_n)$ 
(and hence $\alpha(v_1),\dots,\alpha(v_n)$ is a strictly decreasing arithmetic progression).
\end{itemize} 
\end{lemma} 

\begin{proof} 
For $i\in \{2,3,\dots,n-1\}$,  all edges of $\ugqui$ adjacent to $v_i$ belong to $\Delta$, therefore 
$i$ is a free vertex in $\ugqui^{\alpha,+}$. Thus by Lemma~\ref{lemma:free vertex} 
\[0=(\alpha,\varepsilon_{v_i})_{\ugqui}=(\alpha\vert_{\Delta},\varepsilon_{v_i})_{\Delta}
=2\alpha(v_i)-\alpha_{v_{i-1}}-\alpha_{v_{i+1}},\]  
and (i) follows. Moreover, if $v_n$ has degree $1$ in $\ugqui$, then again by Lemma~\ref{lemma:free vertex} 
\[0=(\alpha,\varepsilon_{v_n})_{\ugqui}=(\alpha\vert_{\Delta},\varepsilon_{v_n})_{\Delta}
=2\alpha(v_n)-\alpha_{v_{n-1}},\]  
showing (ii). 
\end{proof} 

For a subset $S$ of $\ugqui^{\alpha,+}_0$ set 
\[\mu_{\ugqui,\alpha,S}:=\max\{\alpha(v)\mid v\in S\mbox{ is tied}\}.\] 
For $S=\ugqui^{\alpha,+}_0$ we set 
\[\mu_{\ugqui,\alpha}:=\mu_{\ugqui,\alpha,\ugqui^{\alpha,+}_0}.\] 

\begin{lemma}\label{lemma:maxdim} 
For any $\Delta\in \{\Delta^{(1)},\dots,\Delta^{(\kappa_{\ugqui,\alpha})}\}$ 
(where the graphs $\Delta^{(i)}$ are as in Lemma~\ref{lemma:number of Gamma^i}),  denoting by $\Delta_0$ its vertex set, we have the inequality
\[\max\{\alpha(v)\mid v\in \Delta_0\}\le 
\begin{cases} 
\mu_{\ugqui,\alpha,\Delta_0} &\text{ if }\Delta\cong A_n \ (n\ge 1)\\
2\mu_{\ugqui,\alpha,\Delta_0}-1 &\text{ if } \Delta\cong D_n \ (n\ge 4) \\
3\mu_{\ugqui,\alpha,\Delta_0}
&\text{ if }\Delta\cong E_6,\ E_7 \text{ or }E_8.\end{cases}\]
In particular,  
for any $w\in \ugqui^{\alpha,+}_0$ we have $\alpha(w)\le 3\mu_{\ugqui,\alpha}$. 
\end{lemma}

\begin{proof} 
Take $\Delta\in  \{\Delta^{(1)},\dots,\Delta^{(\kappa_{\ugqui,\alpha})}\}$ containing the vertex $w$. 
Consider first the case when $\Delta\cong A_n$: 
\[
\begin{tikzpicture}[scale=0.6]
\node at (-2,0) {$\Delta$:};
\node at (0,0) {$v_1$}; \node at (3,0) {$v_2$} ; \node at (6,0) {$v_3$}; \node at (7.5,0) {$\cdots$}; \node at (9,0) {$v_{n-1}$}; \node at (12,0) {$v_n$} ; 
\draw (0.3,0)--(2.6,0);  \draw (3.3,0)--(5.6,0); \draw (9.6,0)--(11.6,0); 
\node at (15,0) {$(n\ge 3)$};
\end{tikzpicture} 
\]
If both $v_1$ and $v_n$ are tied vertices, then 
\[\max\{\alpha(v_i)\mid i=1,\dots,n\}=\max\{\alpha(v_1),\alpha(v_n)\}=
\mu_{\ugqui,\alpha,\Delta_0}\] 
by Lemma~\ref{lemma:arithmetic} (i). 
If only one of $v_1$ and $v_n$ is tied, say only $v_1$ is a tied vertex, 
then 
\[\max\{\alpha(v_i)\mid i=1,\dots,n\}=
\alpha(v_1)=\mu_{\ugqui,\alpha,\Delta_0}\] 
by Lemma~\ref{lemma:arithmetic} (ii).  

Suppose next that $\Delta$ is isomorphic to $D_n$ ($n\ge 4$), $E_6$, $E_7$, or $E_8$. 
Denote by $v$ the only degree $3$ vertex of $\Delta$ 
(so $v$ is a free vertex), and denote by $w_1$, $w_2$, $w_3$ the end vertices of the three branches starting at $v$. 
Denote by $u_i$ the vertex connected by an edge to $v$ along the branch between $v$ and $w_i$, $i=1,2,3$. Without loss of generality we may assume that $u_1=w_1$ (i.e. the branch connecting $v$ and $w_1$ consists of a single edge), and 
the branch between $v$ and $w_3$ is longest among the three branches. 
By Lemma~\ref{lemma:free vertex} we have 
\begin{equation}\label{eq:degree 3}
0=(\alpha,\varepsilon_v)_{\ugqui}=(\alpha\vert_{\Delta},\varepsilon_v)_{\Delta}=
2\alpha(v)-\alpha(u_1)-\alpha(u_2)-\alpha(u_3)
\end{equation} 
and 
\begin{equation}\label{eq:2alpha(u_i)}
0=(\alpha,\varepsilon_{u_i})_{\ugqui}\le (\alpha\vert_{\Delta},\varepsilon_{u_i})_{\Delta}
\le 2\alpha(u_i)-\alpha(v).\end{equation} 
It follows from \eqref{eq:degree 3} and \eqref{eq:2alpha(u_i)} 
that $\alpha(v)/2\le \alpha(u_i)\le \alpha(v)$ for $i=1,2,3$, 
implying by Lemma~\ref{lemma:arithmetic} (i) that 
\[\max\{\alpha(z)\mid z\in \Delta_0\}=\alpha(v).\] 
So we need to give an upper bound for $\alpha(v)$ in terms of 
$\mu_{\ugqui,\alpha,\Delta_0}$. 

If $u_i=w_i$ is a tied vertex in $\ugqui^{\alpha,+}$ 
(i.e. the branch of $\Delta$ connecting $v$ and $w_i$ consists of a single edge) then the first inequality in \eqref{eq:2alpha(u_i)} is strict, thus 
\[\max\{\alpha(z)\mid z\in \Delta_0\}=\alpha(v)<2\alpha(u_i)\le 2\mu_{\ugqui,\alpha,\Delta_0}.\]  
It remains to deal with the cases when 
$u_i=w_i$ implies that $w_i$ is a free vertex in $\ugqui^{\alpha,+}$, so in the rest of the proof we we shall assume that this is the case. 

Assume that $\Delta\cong D_n$ for some $n\ge 4$: 
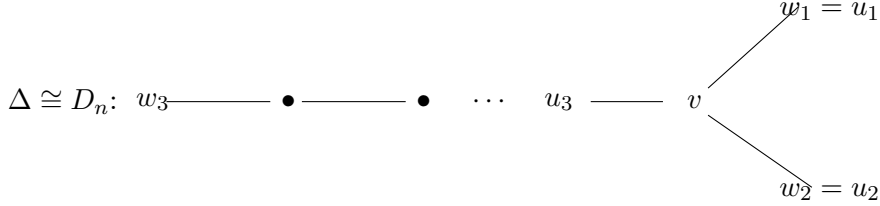
\[
\begin{tikzpicture}[scale=0.6]
\node at (-2,0) {$\Delta\cong D_n$:}; 
\node at (0,0) {$w_3$}; \node at (3,0) {$\bullet$} ; \node at (6,0) {$\bullet$}; \node at (7.5,0) {$\cdots$}; \node at (9,0) {$u_3$}; \node at (12,0) {$v$} ; 
\draw (12.3,0.3)--(14.3,2.1); \draw (12.3,-0.3)--(14.6,-1.9); 
\node at (15,2) {$w_1=u_1$} ; 
\node at (15,-2) {$w_2=u_2$} ; 
\draw (0.3,0)--(2.6,0);  \draw (3.3,0)--(5.6,0); \draw (9.7,0)--(11.3,0);
\end{tikzpicture} 
\] 
Then $u_1=w_1$ and $u_2=w_2$ are free vertices in $\ugqui^{\alpha,+}$,  
$\alpha(u_1)=\alpha(u_2)=\alpha(v)/2$, and hence $\alpha(u_3)=\alpha(v)$.  
Then by Lemma~\ref{lemma:arithmetic} (i)  the value of $\alpha$ along the branch  connecting $v$ and $w_3$ must be constantly 
$\alpha(v)$, and $w_3$ has to be a tied vertex, implying the inequality 
\[\max\{\alpha(z)\mid z\in \Delta_0\}=\alpha(v)=\alpha(w_3)\le \mu_{\ugqui,\alpha,\Delta_0}.\] 

Assume finally that $\Delta$ is $E_6$, $E_7$, or $E_8$. 
As we pointed out before, we may assume that  $w_1$ (the end vertex of the length one branch of $\Delta$) is a free vertex in $\ugqui^{\alpha,+}$.  
Setting $d:=\alpha(w_1)$ we have  $\alpha(v)=2d$, and then 
$\alpha(u_2)=\frac 32d-e$ and $\alpha(u_3)=\frac 32d+e$ for some $e$. This determines the value of $\alpha$ on the remaining vertices of $\Delta$ as it is indicated in the picture below 
for the case $\Delta\cong E_8$. 

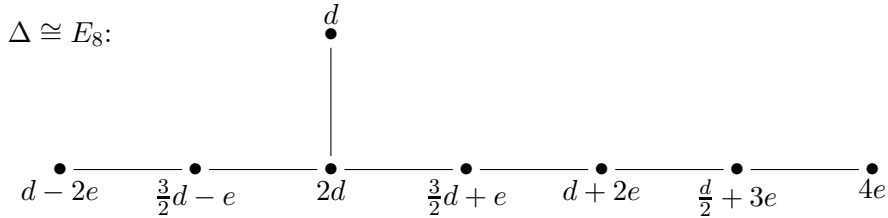
\[
\begin{tikzpicture}[scale=0.6]
\node at (0,3) {$\Delta\cong E_8$:}; 
\node at (0,0) {$\bullet$}; \node at (3,0) {$\bullet$} ; \node at (6,0) {$\bullet$}; 
\node[below] at (0,0) {$d-2e$};  \node[below] at (3,0) {$\frac 32d-e$}; 
\node[below] at (6,0) {$2d$}; \node[above] at (6,3) {$d$}; 
\node[below] at (9,0) {$\frac 32d+e$}; \node[below] at (12,0) {$d+2e$}; 
\node[below] at (15,0) {$\frac d2+3e$}; \node[below] at (18,0) {$4e$}; 
 \node at (6,3) {$\bullet$}; \node at (9,0) {$\bullet$}; \node at (12,0) {$\bullet$}; \node at (15,0) {$\bullet$} ; \node at (18,0) {$\bullet$}; 
\draw (0.3,0)--(2.7,0);  \draw (3.3,0)--(5.7,0); \draw (6.3,0)--(8.7,0); \draw (9.3,0)--(11.7,0);
\draw (12.3,0)--(14.7,0); \draw (15.3,0)--(17.7,0);
\draw (6,0.3)--(6,2.7); 
\end{tikzpicture} 
\] 

Suppose first that $w_2$ (the leftmost vertex on the picture above) is a tied vertex in $\ugqui^{\alpha,+}$. 
Then $2(d-2e)>\frac 32 d-e$, hence $d>6e$. It follows that 
\[\max\{\alpha(z)\mid z\in \Delta_0\}=\alpha(v)=2d<3(d-2e)=3\alpha(w_2)\le 
3\mu_{\ugqui,\alpha,\Delta_0}.\] 
It remains to deal with the case when 
$w_2$ is a free vertex, 
hence $2(d-2e)=\frac 32d-e$, and $w_3$ (the rightmost vertex on the picture above) must be tied. 
Then $d=6e$, and when $\Delta\cong E_8$, the dimension vector $\alpha\vert_{\Delta}$ is a positive integer multiple of 
$\left(\begin{array}{ccccccc} &  & 3 &  &  &  &  \\2 & 4 & 6 & 5 & 4 & 3 & 2\end{array}\right)$.  
So $\alpha(v)/\alpha(w_3)= 6/2=3$, implying 
\[\max\{\alpha(z)\mid z\in \Delta_0\}=\alpha(v)=3\alpha(w_3)\le3\mu_{\ugqui,\alpha,\Delta_0}.\] 
When $\Delta\cong E_7$, the dimension vector $\alpha\vert_{\Delta}$ is a positive integer multiple of 
$\left(\begin{array}{cccccc} &  & 3 &  &  &   \\2 & 4 & 6 & 5 & 4 & 3 \end{array}\right)$, and 
so $\alpha(v)=2\alpha(w_3)$, whereas 
when $\Delta\cong E_6$, the dimension vector $\alpha\vert_{\Delta}$ is a positive integer multiple of 
$\left(\begin{array}{ccccc} &  & 3 &  &    \\2 & 4 & 6 & 5 & 4 \end{array}\right)$, 
whence $\alpha(v)= \frac 32 \alpha(w_3)$. 
So the desired inequality $\alpha(v)\le 3\mu_{\ugqui,\alpha}$ holds in these cases as well.  
\end{proof}  

\begin{lemma}\label{lemma:A_4 subgraph} 
If $|\ugqui^{\alpha,+}_0|>\kappa_{\ugqui,\alpha}\cdot \max\{8,2\mu_{\ugqui,\alpha}+1\}$, 
then $\ugqui^{\alpha,+}$ has a subgraph 
\[
\begin{tikzpicture}[scale=0.6]
\node at (-2,0) {$\Delta$:}; \node at (0,0) {$v_1$}; \node at (3,0) {$v_2$} ; \node at (6,0) {$v_3$};  \node at (9,0) {$v_4$}; 
\draw (0.3,0)--(2.6,0);  \draw (3.3,0)--(5.6,0); \draw (6.3,0)--(8.6,0);
\end{tikzpicture} 
\]
where $\deg_{\ugqui}(v_2)=\deg_{\ugqui}(v_3)=2$ and 
$\alpha(v_1)=\alpha(v_2)=\alpha(v_3)=\alpha(v_4)$. 
\end{lemma} 

\begin{proof} 
Recall that $\ugqui^{\alpha,+}$ is the union of its subgraphs $\Delta^{(1)}$,$\dots$,
$\Delta^{(\kappa_{\ugqui,\alpha})}$ introduced before Lemma~\ref{lemma:number of Gamma^i}. 

\emph{Case 1.: Some $\Delta^{(i)}$ is isomorphic to $A_n$, where 
$n >\max\{3,\mu_{\ugqui,\alpha}\}$.} The values of $\alpha$ along the vertices of 
$\Delta^{(i)}$ are bounded by $\mu_{\ugqui,\alpha}$ by Lemma~\ref{lemma:maxdim}, therefore they can not form a non-constant arithmetic progression of length $n>\mu_{\ugqui,\alpha}$. It follows by  Lemma~\ref{lemma:arithmetic} that $\alpha$ is constant along 
the vertices of $\Delta^{(i)}$, and since $n\ge 4$, the desired statement holds.    

\emph{Case 2.: Some $\Delta^{(i)}$ is isomorphic to $D_n$, where 
$n>\max\{5,2\mu_{\ugqui,\alpha}+1\}$.}  
Then $\Delta^{(i)}$ contains a subgraph isomorphic to $A_{n-2}$ whose  degree $2$ vertices are free, hence have degree $2$ in $\ugqui$. The dimension vector $\alpha$ is smaller than $n-2$ by Lemma~\ref{lemma:maxdim} on the vertices 
of this $A_{n-2}$ subgraph. Thus the values of $\alpha$ must be constant along the 
vertices of this $A_{n-2}$ subgraph by Lemma~\ref{lemma:arithmetic}. 
As $n-2\ge 4$ as well, we are done again. 

Note that $|\ugqui^{\alpha,+}_0|\le \sum_{i=1}^{\kappa_{\ugqui,\alpha}}|\Delta^{(i)}_0|$. 
If $\Delta^{(i)}$ is isomorphic to $E_6$, $E_7$, or $E_8$, then it has at most $8$ vertices. 
Therefore the assumption  
$|\ugqui^{\alpha,+}_0|>\kappa_{\ugqui,\alpha}\cdot  \max\{8,2\mu_{\ugqui,\alpha}+1\}$  
guarantees that necessarily we are in Case 1 or Case 2, 
so the desired statement holds. 
\end{proof} 

\begin{lemma}\label{lemma:shrink in F_Q} 
A $\tau\sigma$-minimal element in the set 
\[\mathcal{C}:=\{(\qui',\alpha') \mid \qui'\mbox{ is  a wild connected quiver, }\alpha'\in F_{\qui'},\ 
\mathrm{supp}(\alpha')=\qui'
\}\] 
does not contain a full subgraph 
\[
\begin{tikzpicture}[scale=0.6]
\node at (0,0) {$v_1$}; \node at (3,0) {$v_2$} ; \node at (6,0) {$v_3$};  \node at (9,0) {$v_4$}; 
\draw (0.3,0)--(2.6,0);  \draw (3.3,0)--(5.6,0); \draw (6.3,0)--(8.6,0);
\end{tikzpicture} 
\]
isomorphic to $A_4$ with vertices $v_1$, $v_2$, $v_3$, $v_4$ such that 
$\deg_{\ugqui}(v_2)=\deg_{\ugqui}(v_3)=2$ and 
$\alpha(v_1)=\alpha(v_2)=\alpha(v_3)=\alpha(v_4)$. 
\end{lemma} 
\begin{proof} 
Let $(\qui,\alpha)$ be a $\tau\sigma$-minimal element in $\mathcal{C}$, 
and suppose for contradiction that $\qui$ has a subgraph as in the statement.  
Denote by $a_i$ the arrow connecting $v_i$ and $v_{i+1}$ for $i=1,2,3$. 
By symmetry we may assume that $\source a_2=v_2$ and $\target a_2=v_3$.  
If $\target a_1=v_2$, then $v_2$ is a large vertex for $(\qui,\alpha)$, and we may apply 
$\tau_{v_2}$ to get the pair 
$\tau_{v_2}(\qui,\alpha)$, where $\tau_{v_2}\qui$ has one less vertices than $\qui$.   
We claim that $\tau_{v_2}(\qui,\alpha)\in \mathcal{C}$, contrary to the assumption that 
$(\qui,\alpha)$ is $\tau\sigma$-minimal in $\mathcal{C}$. 
Clearly $\tau_{v_2}\qui$ is connected and 
$\tau_{v_2}\alpha$ is a sincere dimension vector for $\tau_{v_2}\qui$. 
We show that 
$(\varepsilon_w,\alpha)_{\qui}=(\varepsilon_w,\tau_{v_2}\alpha)_{\tau_{v_2}\qui}$ for any vertex $w\in \tau_{v_2}\qui_1=\qui_1\setminus\{v_2\}$. 
Indeed, a vertex $w\in \tau_{v_2}\qui_1\setminus \{v_1,v_3\}$ is not connected to $v_2$ in $\qui$, hence by \eqref{eq:alphaepsilon} we obviously have $(\varepsilon_w,\tau_{v_2}\alpha)_{\tau_{v_2}\qui}=
(\varepsilon_w,\alpha)_{\qui}$. On the other hand, it is easy to infer from 
\eqref{eq:alphaepsilon} that 
$(\varepsilon_{v_1},\alpha)_{\qui}-(\varepsilon_{v_1},\tau_{v_2}\alpha)_{\tau_{v_2}\qui}=
-\alpha(v_2)+\alpha(v_3)=0$, and 
$(\varepsilon_{v_3},\alpha)_{\qui}-(\varepsilon_{v_3},\tau_{v_2}\alpha)_{\tau_{v_2}\qui}=
-\alpha(v_2)+\alpha(v_1)=0$. Consequently, $\tau_{v_2}\alpha$ belongs to the fundamental set $F_{\tau_{v_2}\qui}$. Since $\qui$ is a wild quiver, it has a vertex $w$ with 
$(\varepsilon_w,\alpha)_{\qui}<0=(\varepsilon_{v_2},\alpha)_{\qui}$,  
so $w\neq v_2$. Thus $w$ is a vertex of $\tau_{v_2}\qui$ with 
$(\varepsilon_w,\tau_{v_2}\alpha)_{\tau_{v_2}\qui}<0$, implying that 
$\tau_{v_2}$ is a wild quiver. 
Therefore $\tau_{v_2}\qui\in \mathcal{C}$, and this contradicts the assumption that  $(\qui,\alpha)$ is $\tau\sigma$-minimal 
in $\mathcal{C}$. 

If $\source a_3=v_3$, then we can apply  
$\tau_{v_3}$ to get a contradiction.  
It remains to deal with the case when $v_2$ is a source and $v_3$ is a sink. 
In this case $v_3$ is a small sink for $(\qui,\alpha)$. Then we may apply the operation $\sigma_{v_3}$. After that the operation $\tau_{v_2}$ can be applied, yielding a contradiction again. 
This finishes the proof. 
\end{proof} 

\bigskip
\begin{proofof}{Theorem~\ref{thm:bounding the quiver}}
Let $(\qui,\alpha)$ be a $\tau\sigma$-minimal element in the set $\mathcal{C}$ 
defined in the statement. 
It follows by Lemma~\ref{lemma:A_4 subgraph} and Lemma~\ref{lemma:shrink in F_Q} 
that 
$|\ugqui^{\alpha,+}_0|\le \kappa_{\ugqui,\alpha}\cdot \max\{8,2\mu_{\ugqui,\alpha}+1\}$. 
We have $\kappa_{\ugqui,\alpha}\le 36(d-1)^2$ by Lemma~\ref{lemma:number of Gamma^i}, 
and $\mu_{\ugqui,\alpha}\le 6(d-1)$ by  Lemma~\ref{lemma:bounds} (v). Moreover, as $\qui$ is wild and $\alpha$ is a sincere dimension vector, we have $\langle\alpha,\alpha\rangle_{\qui}<0$, implying that $d>1$, and hence $\max\{8,12(d-1)+1\}=12d-11$. 
Consequently, $|\ugqui^{\alpha,+}_0|\le  36(d-1)^2(12d-11)$. 
Taking into account Lemma~\ref{lemma:bounds} (i) and we conclude  
\[|\qui_0|=|\ugqui^{\alpha,-}_0|+|\ugqui^{\alpha,+}_0|
\le 2(d-1)+ 36(d-1)^2(12d-11).\] 
As the connected components of the graph $\ugqui^{\alpha,+}$ are trees, 
we have $|\ugqui^{\alpha,+}_1|\le |\ugqui^{\alpha,+}_0|$ 
(with equality when $\qui^{\alpha,+}$ is empty), hence by Lemma~\ref{lemma:bounds} (vii) we have 
\[|\qui_1|\le 2(d-1)+ 36(d-1)^2(12d-11)+12(d-1)^2.\] 
Combining Lemma~\ref{lemma:bounds} (ii), (v) 
we get $\mu_{\qui,\alpha}\le 6(d-1)$, and taking into account Lemma~\ref{lemma:maxdim},  
we conclude  
\[\max\{\alpha(v)\mid v\in \qui_0\}\le 18(d-1).\] 
\end{proofof} 


\section{The derivation of Theorem~\ref{thm:fundamental domain}} \label{sec:moduli}

\begin{proofof}{Corollary~\ref{cor:main}} 
We introduce a strict partial order on $\mathcal{C}$ as follows. For 
$(\qui,\alpha)$ and $(\qui',\alpha')$ in $\mathcal{C}$ we set 
$(\qui',\alpha')\prec (\qui,\alpha)$ if there exists a sequence 
$(\qui^{(i)},\alpha^{(i)})$ ($i=0,\dots,n$) 
of quiver-dimension vector pairs satisfying (1), (2), (3), (4) from 
Definition~\ref{def:tausigma-minimal} with $(\qui',\alpha')=(\qui^{(n)},\alpha^{(n)})$.  
Condition (4) ensures that $\prec$ is a strict partial order satisfying the descending chain condition, so any subset of $\mathcal{C}$ has a minimal element with respect to $\prec$. 

Now assume that the $d$-dimensional algebraic variety $X$ is isomorphic to 
$\moduli(\qui,\alpha,\theta)$, where $(\qui,\alpha)\in \mathcal{C}$ and 
$\alpha$ is $\theta$-stable. Then  there exists a $(\qui',\alpha')$ in $\mathcal{C}$ which is 
\begin{itemize} 
\item minimal with respect to $\prec$; 
\item $(\qui',\alpha')\preceq (\qui,\alpha)$. 
\end{itemize} 
In particular, $(\qui',\alpha')$ is $\tau\sigma$-minimal in $\mathcal{C}$. 
Moreover, a repeated application of Theorem~\ref{thm:main} implies that there exists a weight $\theta'$ for $\qui'$ such that $\alpha'$ is $\theta'$-stable, and 
$\moduli(\qui',\alpha',\theta')\cong \moduli(\qui,\alpha,\theta)$. 
\end{proofof} 

We shall use the following known fact: 

\begin{lemma}\label{lemma:fixed dim, varying weight} 
For a fixed quiver $\qui$ and a fixed dimension vector $\alpha$,  
up to isomorphism there are only finitely many possible moduli spaces 
$\moduli(\qui,\alpha,\theta)$. 
\end{lemma} 

\begin{proof} 
For a weight $\theta\in \mathbb{Z}^{\qui_0}$ set 
$\theta_0:=\{\beta\in \mathbb{N}^{\qui_0}\mid \beta\le \alpha,\ 
\sum_{v\in \qui_0}\theta(v)\beta(v)=0\}$ 
and 
$\theta_-:=\{\beta\in \mathbb{N}^{\qui_0}\mid \beta\le \alpha,\ 
\sum_{v\in \qui_0}\theta(v)\beta(v)<0\}$.  
Now $(\theta_0,\theta_-)=(\theta'_0,\theta'_-)$ implies that 
$\rep(\qui,\alpha)^{\theta\tiny{\text{-sst}}}=\rep(\qui,\alpha)^{\theta'\tiny{\text{-sst}}}$
and 
$\rep(\qui,\alpha)^{\theta\tiny{\text{-st}}}=\rep(\qui,\alpha)^{\theta'\tiny{\text{-st}}}$.  
Therefore 
$(\theta_0,\theta_-)=(\theta'_0,\theta'_-)$ implies that 
\[\moduli(\qui,\alpha,\theta)=\rep(\qui,\alpha)^{\theta\tiny{\text{-sst}}}/\negmedspace/\GL(\alpha)
=\rep(\qui,\alpha)^{\theta'\tiny{\text{-sst}}}/\negmedspace/\GL(\alpha)
=\moduli(\qui,\alpha,\theta').\]  
As there are only finitely many sets of dimension vectors $\beta$ with $\beta\le\alpha$ for a fixed $\alpha$, there are only finitely many possible pairs of sets of the form $(\theta_0,\theta_-)$, where 
$\theta\in \mathbb{Z}^{\qui_0}$. 
\end{proof} 

\bigskip
\begin{proofof}{Corollary~\ref{cor:sketch}}
Assume that  $X$ is a $d$-dimensional algebraic variety, and 
$X\cong \moduli(\qui,\alpha,\theta)$, where $(\qui,\alpha)\in \mathcal{C}$ and 
$\alpha$ is $\theta$-stable. We have $d=1-\langle \alpha,\alpha\rangle_{\qui}$ by 
\eqref{eq:dim of stable moduli}. 
Moreover, by Corollary~\ref{cor:main} we may assume that $(\qui,\alpha)$ is 
$\tau\sigma$-minimal, so it belongs to the 
set 
\[\mathcal{D}:=\{(\qui',\alpha')\in \mathcal{C}\mid 
1-\langle \alpha',\alpha'\rangle_{\qui'}=d \mbox{ and } (\qui',\alpha')\mbox{ is }\tau\sigma\mbox{-minimal in }\mathcal{C}\}.\] 
By assumption the set $\mathcal{D}$ is finite.  

 Thus $X$ is isomorphic to a moduli space associated with some element of the finite set 
$\mathcal{D}$. By Lemma~\ref{lemma:fixed dim, varying weight} there are only finitely many such varieties, so we are done with the proof. 
\end{proofof}

\bigskip 
\begin{proofof}{Theorem~\ref{thm:fundamental domain}} 
Let $X$ be an algebraic variety with $\dim(X)\le d$ and $X\cong \moduli(\qui,\alpha,\theta)$ where $\qui$ is $\alpha$-stable and $\alpha\in F_{\qui}$. As we pointed out 
in Section~\ref{subsec:quivrep}, we may assume that $\alpha$ is a sincere dimension vector for $\qui$. Then both the $\theta$-stability of $\alpha$ or $F_{\ugqui}\neq\emptyset$ imply  that $\ugqui$ is connected. If $\qui$ is an extended Dynkin quiver, then  $\moduli(\qui,\alpha,\theta)$ is isomorphic to a projective or affine space by 
 \cite[Theorem 3.1]{domokos:gmj} (see also  \cite{skowronski-weyman}, or 
\cite{domokos-lenzing:1}, \cite{domokos-lenzing:2}). So $X$ belongs to a finite set 
$\mathcal{X}'$ of algebraic varieties in this case. 
Otherwise $\qui$ is wild, and $X$ is isomorphic to an element of $\mathcal{X}''$, where 
$\mathcal{X}''$ is a complete set of representatives of the isomorphism classes of the varieties $\moduli(\qui',\alpha',\theta')$, with $\alpha'$ a $\theta'$-stable dimension vector and $(\qui',\alpha')$ belonging to 
\[\mathcal{C}:=\{(\qui',\alpha') \mid \qui'\mbox{ is  a wild connected quiver, }\alpha'\in F_{\qui'},\ 
\mathrm{supp}(\alpha')=\qui'
\}.\] 
The condition of Corollary~\ref{cor:sketch} holds for $\mathcal{C}$ by 
Theorem~\ref{thm:bounding the quiver}. 
Therefore by Corollary~\ref{cor:sketch}, $\mathcal{X}''$ is finite. 
Thus $X$ is isomorphic to a variety in the finite set $\mathcal{X}'\cup \mathcal{X}''$, 
finishing the proof. 
\end{proofof} 

\section{Products of quiver moduli spaces}\label{sec:products} 

\begin{proofof}{Proposition~\ref{prop:product}} 
Let $X$ be a projective algebraic variety such that $X\cong \moduli(\qui,\alpha,\theta)$. 
We may assume that $\alpha$ is a $\theta$-semistable 
sincere dimension vector for $\qui$.   Then necessarily the quiver $\qui$ is acyclic. 
Note that in this case a general $\alpha$-dimensional representation of $\qui$ is $\theta$-semistable.
Following \cite[Definition 3.10]{derksen-weyman} we say that 
\begin{equation}\label{eq:theta-stable decomposition}
\alpha=m_1\cdot\alpha_1\dot{+}m_2\cdot \alpha_2\dot{+}\cdots\dot{+}m_s\cdot\alpha_s
\end{equation} is 
\emph{the $\theta$-stable decomposition of}  $\alpha$  if a general $\alpha$-dimensional representation of $\qui$ has a composition series in the category of $\theta$-semistable 
representations with $m_i$ composition factors of dimension vector $\alpha_i$ for $i=1,\dots,s$ (so the $\alpha_i$ are $s$ distinct $\theta$-stable dimension vectors for $\qui$). 
As a special case of \cite[Theorem 1.4]{chindris} (dealing with quivers with relations and 
inspired by \cite[Theorem 3.16]{derksen-weyman}), \eqref{eq:theta-stable decomposition}  implies  
\begin{equation}\label{eq:symm product} 
\moduli(\qui,\alpha,\theta)\cong \mathrm{S}^{m_1}(\moduli(\qui,\alpha_1,\theta))\times \cdots 
\times \mathrm{S}^{m_s}(\moduli(\qui,\alpha_s,\theta)),\end{equation} 
where $\mathrm{S}^m(Y)$ stands for the $m$th symmetric power of the projective algebraic variety $Y$.   
If $\alpha_i$ in \eqref{eq:theta-stable decomposition} is a real root, then $\moduli(\qui,\alpha_i,\theta)$ is a point, and the factor $\mathrm{S}^{m_i}(\moduli(\qui,\alpha_i,\theta))$ can be omitted in \eqref{eq:symm product}. 
If $\alpha_i$ in \eqref{eq:theta-stable decomposition} is an isotropic imaginary root, 
then $\mathrm{supp}(\alpha_i)$ is an extended Dynkin quiver. 
It is well known (see \cite[Theorem 3.1]{domokos:gmj}, or \cite{skowronski-weyman}, or 
\cite{domokos-lenzing:1}, \cite{domokos-lenzing:2}) that either 
$\moduli(\qui,\alpha_i,\theta)$ is a point 
(and then the corresponding factor in \eqref{eq:symm product} can be omitted), 
or $\moduli(\qui,\alpha_i,\theta)$ is a projective line, and hence 
$\mathrm{S}^{m_i}(\moduli(\qui,\alpha_i,\theta))\cong \mathbb{P}^{m_i}$. 
If $\alpha_i$ is imaginary nonisotropic, then $\mathrm{supp}(\alpha_i)$ is a connected wild quiver, and by \cite[Proposition 3.15(b)]{derksen-weyman}, we have $m_i=1$. 
As we noted at the end of Section~\ref{subsec:quivrep}, 
$\moduli(\qui,\alpha_i,\theta)\cong \moduli(\mathrm{supp}(\alpha_i),\alpha_i\vert_{\mathrm{supp}(\alpha_i)},\theta\vert_{\mathrm{supp}(\alpha_i)})$, and 
$\alpha_i\vert_{\mathrm{supp}(\alpha_i)}$ is a $\theta\vert_{\mathrm{supp}(\alpha_i)}$-stable dimension vector for the acyclic quiver $\mathrm{supp}(\alpha_i)$. 
This finishes the proof of the "only if" statement. 

The "if" statement follows by Lemma~\ref{lemma:disjoint union} below and the well-known fact 
that $\moduli(\qui,(n,n),(-1,1))\cong \mathbb{P}^n$, where $\qui$ is the Kronecker quiver with 
two vertices and two arrows pointing from the first vertex to the second.  
\end{proofof} 

In the above proof we used the following obvious statement relating the case when  
our quiver $\qui$ is the disjoint union of the quivers 
$\qui'$ and $\qui''$. For a sincere dimension vector 
$\alpha\in \mathbb{N}^{\qui_0}$ and for a weight $\theta\in \mathbb{Z}^{\qui_0}$ set 
$\alpha':=\alpha\vert_{\qui'_0}$, $\alpha'':=\alpha\vert_{\qui''_0}$, 
$\theta':=\theta\vert_{\qui'_0}$, $\theta'':=\theta\vert_{\qui''_0}$. 

\begin{lemma} \label{lemma:disjoint union} 
Let $\qui$, $\alpha$, $\theta$ be as above. 
Then $\alpha$ is $\theta$-semistable for  $\qui$ if and only if 
$\alpha'$ is $\theta'$-semistable for $\qui'$ and $\alpha''$ is $\theta''$-semistable  for 
$\qui''$. Moreover, in this case the variety 
$\moduli(\qui,\alpha,\theta)$ is isomorphic to the product of the varieties 
$\moduli(\qui',\alpha',\theta')$ and $\moduli(\qui'',\alpha'',\theta'')$. 
\end{lemma} 


\section{Some consequences of Theorem~\ref{thm:fundamental domain}}
\label{sec:consequences} 

\begin{proposition}\label{prop:stable decomp for 2 or n1}
Let $\qui$ be an acyclic quiver, $\alpha$ a dimension vector for $\qui$ satisfying condition 
(i) (respectively (ii)) in the statement of Proposition~\ref{prop:2 or n1}, and let $\theta$ be a weight for $\qui$ such that $\alpha$ is $\theta$-semistable. Consider the 
$\theta$-stable decomposition  (cf. the proof of Proposition~\ref{prop:product}) 
\[\alpha=m_1\cdot\alpha_1\dot{+}m_2\cdot \alpha_2\dot{+}\cdots\dot{+}m_s\cdot\alpha_s\] 
of $\alpha$. Then the pairs $(\mathrm{supp}(\alpha_i),\alpha_i)$ also satisfy condition (i) 
(respectively (ii)) of  Proposition~\ref{prop:2 or n1}. 
\end{proposition} 

\begin{proof} 
If condition (i) of  Proposition~\ref{prop:2 or n1} holds for $(\qui,\alpha)$, then the same obviously hold for the $\alpha_i$, since we have $\alpha=\sum_{i=1}^sm_i\alpha_i$. 
Suppose next that condition (ii) of  Proposition~\ref{prop:2 or n1} holds for $(\qui,\alpha)$. 
Denote by $\underline{1}$ the dimension vector taking value $1$ on each of the vertices, 
so $\alpha=n\underline{1}$ for some positive integer $n$. Since $\alpha$ is $\theta$-semistable, we have $\sum_{v\in \qui_0}\theta(v)=0$. 
We claim that $\underline{1}$ is $\theta$-semistable. Indeed, suppose for contradiction that  there exists a 
dimension vector $\beta$ with $\beta\le \underline{1}$ such that 
$\sum_{v\in\qui_0}\theta(v)\beta(v)<0$, and all representations of $\qui$ 
with dimension vector $\underline{1}$ have a subrepresentation with dimension vector 
$\beta$. This holds in particular for the representation assigning $1$ to each arrow of $\qui$. It follows that $\mathrm{supp}(\beta)$ is successor-closed (i.e. if 
$\source a\in \mathrm{supp}(\beta)_0$ for some $a\in \qui_1$, then 
$\target a  \in \mathrm{supp}(\beta)_0$). This implies that all representations of 
dimension vector $\alpha=n\underline{1}$ have a subrepresentation with dimension vector 
$n\beta$, and $\sum_{v\in\qui_0}\theta(v)(n\beta)(v)=n\sum_{v\in\qui_0}\theta(v)\beta(v)<0$, 
contrary to the assumption that $\alpha$ is $\theta$-semistable. 
So $\underline{1}$ is $\theta$-semistable, and we can take its $\theta$-stable decomposition 
\[\underline{1}=\gamma_1\dot{+}\cdots\dot{+}\gamma_t.\] 
By \cite[Proposition 3.14]{derksen-weyman} the $\theta$-stable decomposition of 
$n\underline{1}$ is 
\begin{equation}\label{eq:n underline 1}n\underline{1}=\{n\gamma_1\}\dot{+}\cdots\dot{+}\{n\gamma_s\}\end{equation} 
where 
\[\{n\beta\}=\begin{cases} \underbrace{\beta\dot{+}\cdots\dot{+}\beta}_n 
\text{ if }\beta\text{ is real or imaginary isotropic};\\
n\beta\text{ if }\beta\text{ is imaginary nonisotropic}. 
\end{cases}\] 
Now $(\mathrm{supp}(\gamma_i),\gamma_i)$ and $(\mathrm{supp}(\gamma_i),n\gamma_i)$ clearly satisfy (ii) of Proposition~\ref{prop:2 or n1}. 
\end{proof} 

\begin{proposition}\label{prop:FQ 2 or n1} 
If $(\qui,\alpha)$ is $\tau\sigma$-minimal in the set $\mathcal{C}$ of sincere 
quiver-dimension vector pairs satisfying (i) or (ii) of Proposition~\ref{prop:2 or n1}, and $\qui$ is connected, then $\alpha\in F_{\qui}$ or $\qui_1=\emptyset$.  
\end{proposition} 
\begin{proof} Suppose that $\qui_1$ is non-empty, $(\qui,\alpha)$ is $\tau\sigma$-minimal in $\mathcal{C}$, and $\qui$ is connected. 
 
Assume first that $(\qui,\alpha)$ satisfies condition (ii) of Proposition~\ref{prop:2 or n1} and is $\tau\sigma$-minimal in $\mathcal{C}$. 
If $v\in \qui_0$ has degree $1$, then $v$ is a large vertex for $(\qui,\alpha)$. 
Applying $\tau_v$ we would get a pair in $\mathcal{C}$ with one less arrow, so $(\qui,\alpha)$ would not be $\tau\sigma$-minimal. Therefore all vertices of the connected quiver $\qui$ have  
degree at least two. Since $\qui$ is connected, the assumption on the dimension vector $\alpha$ implies that $\alpha\in F_{\qui}$. 

Assume next that $(\qui,\alpha)$ satisfies condition (i) of Proposition~\ref{prop:2 or n1}.  
Suppose for contradiction that $\qui$ has 
a vertex $u$ such that $(\alpha,\varepsilon_u)_{\ugqui}>0$. It follows that there is no loop at $u$, and $\deg_{\qui}(u)\ge 1$ (by connectedness of $\qui$).  

If $\alpha(u)=1$, then necessarily $\deg_{\ugqui}(u)=1$, and $\alpha(w)=1$ for  
the only vertex $w$ connected to $u$ by an edge in $\ugqui$. So $u$ is a large vertex in 
$(\qui,\alpha)$, and $\tau_u(\qui,\alpha)$ belongs to $\mathcal{C}$ with 
$|\tau_u\qui_0|<|\qui_0|$, contrary to the $\tau\sigma$-minimality of $(\qui,\alpha)$.  

It remains to deal with the case $\alpha(u)=2$. Then $\deg_{\ugqui}(u)\le 3$ follows from 
$(\alpha,\varepsilon_u)_{\ugqui}>0$.  
If $\deg_{\qui}(u)\in \{1, 2\}$, then $u$ is a large vertex for $(\qui,\alpha)$, or $u$ is a small sink or a small source for $(\qui,\alpha)$. Clearly, 
in the first case $\tau_u(\qui,\alpha)\in \mathcal{C}$ with $|\tau_u\qui_0|<|\qui_0|$, 
whereas in the second case 
$\sigma_u(\qui,\alpha)\in \mathcal{C}$ with $|\sigma_u\qui_0|=|\qui_0|$ and 
$|\sigma_u\alpha|<|\alpha|$ (because $\sigma_u\alpha(u)= 2+1-2=1<2=\alpha(u)$), contrary to the $\tau\sigma$-minimality of $(\qui,\alpha)$. 

Finally, 
if $\deg_{\ugqui}(u)=3$, then the value of $\alpha$ at the end vertices of the three arrows adjacent to $u$ is $1$.  If $u$ is a source or a sink, then we may apply $\sigma_u$ to get 
the pair $\sigma_u(\qui,\alpha)\in \mathcal{C}$ with $\sigma_u\alpha(u)=1+1+1-2=1<\alpha(u)$, contrary to the $\tau\sigma$-minimality of $(\qui,\alpha)$. 
If $u$ is neither a source nor a sink, then it is necessarily a large vertex for $(\qui,\alpha)$, 
for which $\tau_u$ can be applied, and since $\tau_u(\qui,\alpha)\in \mathcal{C}$, this 
contradicts the assumption that $(\qui,\alpha)$ is $\tau\sigma$-minimal. 

Summarizing, we showed that  the connected quiver $\qui$ can not have a 
vertex $u$ with $(\alpha,\varepsilon_u)_{\qui}>0$. That is, $\alpha\in F_{\ugqui}$. 
\end{proof}

\bigskip
\begin{proofof}{Proposition~\ref{prop:2 or n1}}
Let $\qui$ be an acyclic quiver and $\alpha$ a sincere dimension vector  for $\qui$ satisfying (i) or (ii) in the statement of Proposition~\ref{prop:2 or n1}, and let $\theta$ be a weight for $\qui$ such that $\alpha$ is $\theta$-semistable. 
By Proposition~\ref{prop:product} and its proof,  
the moduli space  $\moduli(\qui,\alpha,\theta)$ is isomorphic to a product of projective spaces and varieties of the form $\moduli(\qui,\beta,\theta)$, 
where $\beta$ is a summand in the $\theta$-stable decomposition of $\alpha$, 
and $\mathrm{supp}(\beta)$ is a wild quiver.  
By Proposition~\ref{prop:stable decomp for 2 or n1} the pair 
$(\mathrm{supp}(\beta),\beta)$ belongs to the 
set $\mathcal{C}$ of sincere 
quiver-dimension vector pairs satisfying (i) or (ii) of Proposition~\ref{prop:2 or n1}. 
Moreover, $\mathrm{supp}(\beta)$ is $\theta$-stable, and by Corollary~\ref{cor:main} there exists a $\tau\sigma$-minimal sincere pair $(\qui',\beta')\in \mathcal{C}$ and a weight $\theta'$ such that $\beta'$ is $\theta'$-stable (hence in particular, $\qui'$ is connected), 
and 
\[\moduli(\qui,\beta,\theta)=\moduli(\mathrm{supp}(\beta),\beta,\theta)\cong 
\moduli(\qui',\beta',\theta').\]  
Furthermore, by Proposition~\ref{prop:FQ 2 or n1} we have that $\beta'\in F_{\qui'}$, 
so $\moduli(\qui',\beta',\theta')$ belongs to the set $\mathcal{X}$ defined in Section~\ref{sec:main results}, and the proof is complete.  
\end{proofof}


\section{An example} \label{sec:example} 

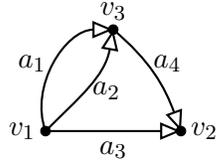
\begin{figure}[h] 
\caption{{\small A $\tau\sigma$-minimal quiver-dimension vector pair not in the fundamental set}} 
\label{figure:1}
\begin{center}
\begin{tikzpicture}[>=open triangle 45,scale=0.9]
\foreach \x in {(0,0),(2,0),(1,1.5)} \filldraw \x circle (2pt); 
\node [left] at (0,0) {$v_1$};
\node [right] at (2,0) {$v_2$};
\node [above] at (1,1.5) {$v_3$};
\node  [right] at (3,1.5) {$\alpha=(\alpha(v_1),\alpha(v_2),\alpha(v_3))=(2,1,3)$};
\node  [right] at (3,0.5) {$\theta=(\theta(v_1),\theta(v_2),\theta(v_3))=(-2,1,1)$};
\node  at (-0.2,1) {$a_1$};
\node at (0.9,0.6) {$a_2$};
\node [below] at (1,0) {$a_3$}; 
\node at (1.8,1) {$a_4$}; 
\draw [->,thick ] (0,0) to [out=110, in=180] (1,1.5); 
\draw [->,thick ] (0,0) to [out=45, in=260] (1,1.5); 
\draw [->,thick ] (0,0) to (2,0); 
\draw [->,thick ] (1,1.5) to [out=320, in =110] (2,0); 
\end{tikzpicture}\quad
\end{center}
\end{figure}

\begin{proposition}\label{prop:example} 
The quiver $\qui$, the dimension vector $\alpha$, and the weight $\theta$ given on Figure~\ref{figure:1} satisfy the following: 
\begin{itemize} 
\item[(i)] $\alpha$ is not contained in the fundamental set $F_{\qui}$. 
\item[(ii)] $\alpha$ is $\theta$-stable. 
\item[(iii)] The pair $(\qui,\alpha)$ can not be transformed by a sequence of operations of the form $\sigma_v$ into a pair $(\qui',\alpha')$ with $\alpha'\in F_{\qui'}$. 
\item[(iv)] The pair $(\qui,\alpha)$ is $\tau\sigma$-minimal among all quiver-dimension vector pairs. 
\end{itemize}  
\end{proposition} 
\begin{proof} 
(i) We have 
$(\alpha,\varepsilon_{v_3})_{\qui}=2\alpha(v_3)-(\alpha(v_1)+\alpha(v_1)+\alpha(v_2))=2\cdot 3-(2+2+1)=1>0$. 

(ii) We have $\sum_{i=1}^3\theta(v_i)\alpha(v_i)=(-2)\cdot 2+1\cdot 1+1\cdot 3=0$. 
Consider the following representation $x=(x_{a_1},x_{a_2},x_{a_3},x_{a_4})\in \rep(\qui,\alpha)$:  
\[x:=( \left(  \begin{matrix} 1 & 0 \\ 0 & 1 \\ 0 & 0\end{matrix}\right), 
\left(  \begin{matrix} 0 & 1 \\ 0 & 0 \\ 1 & 0\end{matrix}\right), 
\left(\begin{matrix} 0 & 1\end{matrix}\right), 
\left(\begin{matrix} 1 & 0  & 0\end{matrix}\right)).\] 
Let $S$ be a proper non-zero subrepresentation of $x$, and write $\beta$ for its dimension vector. We claim that $\sum_{i=1}^3\theta(v_i)\beta(v_i)>0$. 
If $\beta(v_1)=2$ (i.e. $S_{v_1}=x_{v_1}=\F^2$), then the images of $S_{v_1}$ under $x_{a_1}$ and $x_{a_2}$ together span $x_{v_3}=\F^3$, and the image of $S_{v_1}$ under $x_{a_3}$ spans $x_{v_2}$, so $\beta=\alpha$, contrary to our assumption. 
So $\beta(v_1)\le 1$. 
If $\beta(v_1)=0$, then $\sum_{i=1}^3\theta(v_i)\beta(v_i)=
\beta(v_2)+\beta(v_3)>0$. From now on assume that $\beta(v_1)=1$. 
For any non-zero $(y_1\ y_2)^T\in S_{v_1}$ we have that 
$S_{a_1}((y_1\ y_2)^T)=(y_1 \ y_2\ 0)^T$ and 
$S_{a_2}((y_1\ y_2)^T)=(y_2\ 0\ y_1)^T$ are linearly independent and not both are contained in the kernel of $S_{a_4}$, hence $\beta(v_3)\ge 2$ and $\beta(v_2)=1$. 
Thus  $\sum_{i=1}^3\theta(v_i)\beta(v_i)\ge (-2)\cdot 1+1\cdot 2+1\cdot 1>0$, finishing the proof of our claim. Therefore $x$ is a $\theta$-stable representation of $\qui$ of dimension vector $\alpha$,  and so (ii) is proved. 

(iii) and (iv): Suppose for contradiction that  
$(\qui^{(i)},\alpha^{(i)})$ $(i=0,1,\dots,n)$ is a sequence of quiver-dimension vector pairs 
such that $(\qui^{(i+1}),\alpha^{(i+1)})=\sigma_{u_i}(\qui^{(i)},\alpha^{(i)})$ for $i=0,\dots,n-1$ and 
$\alpha^{(n)}\in F_{\qui^{(n)}}$, or this sequence satisfies the conditions 
given in Definition~\ref{def:tausigma-minimal}.  
The pair $(\qui,\alpha)$ has no large vertex, whereas $v_1$ is a small source and $v_2$ is a small sink. Therefore $(\qui^{(1)},\alpha^{(1)})=\sigma_{v_2}(\qui,\alpha)$ or 
$(\qui^{(1)},\alpha^{(1)})=\sigma_{v_1}(\qui,\alpha)$. 

Case I.: $(\qui^{(1)},\alpha^{(1)})=\sigma_{v_2}(\qui,\alpha)$. 
Note that the dimension vector $\alpha$ satisfies the following inequalitites: 
\begin{align} 
\label{eq:1} (\varepsilon_{v_2},\alpha)_{\qui}<0 \\
\label{eq:2} 2\alpha(v_1) >\alpha(v_3),\quad \alpha(v_3)>\alpha(v_2),\quad \alpha(v_1)>\alpha(v_2), 
\quad (\varepsilon_{v_1},\alpha)_{\qui}<0
\end{align} 
By \eqref{eq:1} we have that 
$(\varepsilon_{v_2},\alpha^{(1)})_{\qui^{(1)}}>0$, so 
$\alpha^{(1)}\notin F_{\qui^{(1)}}$, moreover, $|\alpha^{(1)}|>|\alpha|$. 
The inequalities \eqref{eq:2} imply that the pair $(\qui^{(1)},\alpha^{(1)})$ 
has no large vertex, and $v_3$ is a small sink for $(\qui^{(1)},\alpha^{(1)})$ with $(\varepsilon_{v_3},\alpha^{(1)})_{\qui^{(1)}}<0$. So we must have 
$(\qui^{(2)},\alpha^{(2)})=\sigma_{v_3}(\qui^{(1)},\alpha^{(1)})$, and so 
$(\varepsilon_{v_3},\alpha^{(2)})_{\qui^{(2)}}>0$. 
Thus $\alpha^{(2)}\notin F_{\qui^{(2)}}$, and $|\alpha^{(2)}|>|\alpha^{(1)}|$. 
The inequalitites \eqref{eq:2} imply that 
the pair $(\qui^{(2)},\alpha^{(2)})$ 
has no large vertex, and $v_1$ is a small sink for $(\qui^{(2)},\alpha^{(2)})$ with $(\varepsilon_{v_1},\alpha^{(2)})_{\qui^{(1)}}<0$. 
Therefore we must have 
$(\qui^{(3)},\alpha^{(3)})=\sigma_{v_1}(\qui^{(2)},\alpha^{(2)})$, and so 
$(\varepsilon_{v_1},\alpha^{(3)})_{\qui^{(3)}}>0$. 
Thus $\alpha^{(3)}\notin F_{\qui^{(3)}}$, and $|\alpha^{(3)}|>|\alpha^{(2)}|$. 
Note that the quiver $\qui^{(3)}$ is the same as $\qui$, and the inequalities 
\eqref{eq:1}, \eqref{eq:2} hold for $\alpha^{(3)}$ instead of $\alpha$ 
(in partcular, there is no large vertex for $(\qui^{(3)},\alpha^{(3)})$).  
So the above argument continues, and we can never stop. This is a contradiction. 

Case II.:  $(\qui^{(1)},\alpha^{(1)})=\sigma_{v_1}(\qui,\alpha)$. 
The argument is similar to the argument for Case I. 
We shall use the following inequalities satisfied by $\alpha$:  
\begin{align} 
\label{eq:3} (\varepsilon_{v_1},\alpha)_{\qui}<0 \\
\label{eq:4} 
\alpha(v_3) >\alpha(v_1),\quad 2\alpha(v_3)>\alpha(v_2) 
\end{align} 
Then \eqref{eq:3} implies that $\alpha^{(1)}\notin F_{\qui^{(1)}}$ and 
$|\alpha^{(1)}|>|\alpha|$. Inequalities  \eqref{eq:4} imply that 
$\alpha^{(1)}(v_2)<\alpha^{(1)}(v_1)$, hence $v_2$ is not a large vertex for 
$(\qui^{(1)},\alpha^{(1)})$, and $(\varepsilon_{v_3},\alpha^{(1)})_{\qui^{(1)}}<0$. 
It follows that 
$(\qui^{(2)},\alpha^{(2)})=\sigma_{v_3}(\qui^{(1)},\alpha^{(1)})$, 
$\alpha^{(2)}\notin F_{\qui^{(2)}}$, and $|\alpha^{(2)}|>|\alpha^{(1)}|$. 
Now inequalities  \eqref{eq:4} imply that 
$(\varepsilon_{v_2},\alpha^{(2)})_{\qui^{(2)}}<0$ and 
$(\varepsilon_{v_3},\alpha^{(2)})_{\qui^{(2)}}<0$. 
Thus there is no large vertex for $(\qui^{(2)},\alpha^{(2)})$, 
and $(\qui^{(3)},\alpha^{(3)})=\sigma_{v_2}(\qui^{(2)},\alpha^{(2)})$, 
$\alpha^{(3)}\notin F_{\qui^{(3)}}$, $|\alpha^{(3)}|>|\alpha^{(2)}|$. 
Note that $\qui^{(3)}=\qui$, and it is easy to deduce from 
\eqref{eq:3} and \eqref{eq:4} that these inequalities hold for 
$\alpha^{(3)}$ instead of $\alpha$, and moreover, 
$v_3$ is not a large vertex for $(\qui^{(3)},\alpha^{(3)})$. 
Thus we must have $(\qui^{(4)},\alpha^{(4)})=\sigma_{v_1}(\qui^{(3)},\alpha^{(3)})$.  
the process continues in the same way as above, and we can never stop. This is a contradiction, and the proof is complete. 
\end{proof} 

\section{Remarks on affine quotients} \label{sec:affine}

For the weight $\theta=0$ the moduli space $\moduli(\qui,\alpha,0)$ is the affine quotient 
$\rep(\qui,\alpha)/\negmedspace/\GL(\alpha)$. 
It is an affine variety whose points are in a natural bijection with the semisimple representations of $\qui$ of dimension vector $\alpha$ (see \cite{lebruyn-procesi}). 
It is just a point when $\qui$ is acyclic, so 
we get interesting examples only if we allow the quiver  to have oriented cycles. 
The following finiteness statement follows from the methods and results of \cite{bocklandt-lebruyn-deweyer} (although it is not explicitly stated there): 

\begin{proposition}\label{prop:affine finite} 
For any positive integer $d$ there are only finitely many $d$-dimensional affine varieties that are isomorphic 
to $\moduli(\qui,\alpha,0)$ for some quiver $\qui$ and dimension vector $\alpha$. 
\end{proposition} 
 
A consequence of Proposition~\ref{prop:affine finite} by \cite{adriaenssens-lebruyn} is that the $d$-dimensional quasi-projective moduli spaces 
$\moduli(\qui,\alpha,\theta)$ (where $\theta$ is arbitrary)  may have only finitely many possible singularities (up to \'etale isomorphism). 
We give a short proof of Proposition~\ref{prop:affine finite}. 
A quiver $\qui$ is \emph{strongly connected} if for any pair $v,w\in \qui_0$, there is a directed path from $v$ to $w$. 

\bigskip
\begin{proofof}{Proposition~\ref{prop:affine finite}} 
It follows from the description of the generators of $\F[\qui,\alpha]^{\GL(\alpha)}$ 
in \cite{lebruyn-procesi}, \cite{donkin} that any moduli space $\moduli(Q,\alpha,0)$ is isomorphic to a product of such moduli spaces where the quiver is a strongly connected subquiver of $\qui$ and the dimension vector is sincere. When $\qui$ is a one-lop quiver then $\moduli(\qui,n,0)$ is the affine $n$-space for some $n$. 
From now on we assume that our quiver is not the one-loop quiver. 
The reduction step R1 from \cite[Lemma 2.4]{bocklandt-smooth} shows that if $\langle \alpha,\varepsilon_v\rangle_{\qui} \ge 0$ or $\langle \varepsilon_v,\alpha\rangle_{\qui}\ge 0$ for some $v\in \qui_0$, then $\moduli(\qui,\alpha,0)$ is isomorphic to a moduli space associated to a quiver with one less vertices. Therefore it is sufficient to show that for each positive integer $d$ there are 
finitely many pairs $(\qui,\alpha)$ where $\qui$ is strongly connected (and is not the one-loop quiver), $\alpha$ is a sincere dimension vector for $\qui$ with 
$\langle \alpha,\varepsilon_v\rangle_{\qui} < 0$ and  $\langle \varepsilon_v,\alpha\rangle< 0$ for all $v\in \qui_0$, 
and $\dim \moduli(\qui,\alpha,0)=d$.  
Now let $(\qui,\alpha)$ be such a pair. 
Then the general $\alpha$-dimensional representation of $\qui$ is simple 
by \cite[Theorem 4]{lebruyn-procesi}, hence 
$d=1-\langle\alpha,\alpha\rangle_{\qui}$, implying that 
\begin{equation}\label{eq:bound for vertices}
d-1=-\langle \alpha,\alpha\rangle_{\qui}=-\sum_{v\in \qui_0}\alpha(v)\langle \varepsilon_v,\alpha\rangle_{\qui}\ge\sum_{v\in\qui_0}\alpha(v).\end{equation}  
Denote by $\underline{1}$ the constant $1$ dimension vector for our $\qui$, so $\underline{1}\le \alpha$. 
The decription of the $\GL$-invariants on quiver representation spaces implies that 
$\moduli(\qui,\underline{1},0)$ can be embedded into $\moduli(\qui,\alpha,0)$.  Moreover, 
$\underline{1}$ is a simple dimension vector (since it is a sincere dimension vector for the strongly connected quiver $\qui$). 
Thus we have  
\begin{equation}\label{eq:bound for edges} 
d=\dim \moduli(\qui,\alpha,0)\ge \dim \moduli(\qui,\underline{1},0)=
1-\langle \underline{1},\underline{1}\rangle_{\qui}=1-|\qui_0|+|\qui_1|.\end{equation} 
Thus we have $|\qui_0|\le d-1$ and $|\alpha|\le d-1$ by \eqref{eq:bound for vertices} 
and hence $|\qui_1|\le 2(d-1)$ by \eqref{eq:bound for edges}. 
Obviouly there are finitely many such pairs $(\qui,\alpha)$ for fixed $d$. 
\end{proofof}

 \begin{center} Acknowledgement \end{center} 
 
I thank D\'aniel Jo\'o for discussions related to the material in this paper, 
and in particular for simplifications in the proof of Proposition~\ref{prop:affine finite}. 
I am also grateful to the referee for constructive criticism that helped me to 
simplify and improve the presentation of this material. 


\end{document}